\documentclass[11pt]{amsart}
\allowdisplaybreaks

\usepackage{amssymb
,amsthm
,amsmath
,amscd
,mathtools	
}
\usepackage[all]{xy}
\usepackage{url}
\usepackage[dvipdfmx]{graphicx}
\usepackage[top=30truemm,bottom=30truemm,left=25truemm,right=25truemm]{geometry}

\newcommand{\A}{\mathbb{A}}
\newcommand{\F}{\mathbb{F}}
\newcommand{\Z}{\mathbb{Z}}
\newcommand{\Q}{\mathbb{Q}}
\newcommand{\R}{\mathbb{R}}
\newcommand{\C}{\mathbb{C}}

\renewcommand{\AA}{\mathcal{A}}
\newcommand{\DD}{\mathcal{D}}
\newcommand{\LL}{\mathcal{L}}
\newcommand{\MM}{\mathcal{M}}
\newcommand{\PP}{\mathcal{P}}
\newcommand{\Sc}{\mathcal{S}}
\newcommand{\UU}{\mathcal{U}}
\newcommand{\Ha}{\mathfrak{H}}

\newcommand{\fin}{\mathrm{fin}}
\newcommand{\Ind}{\mathrm{Ind}}
\newcommand{\Irr}{\mathrm{Irr}}
\newcommand{\re}{\mathrm{Re}}
\newcommand{\im}{\mathrm{Im}}
\newcommand{\std}{\mathrm{std}}
\newcommand{\spin}{\mathrm{spin}}
\newcommand{\tr}{\mathrm{tr}}
\newcommand{\sgn}{\mathrm{sgn}}
\newcommand{\unit}{\mathrm{unit}}
\newcommand{\Gal}{\mathrm{Gal}}
\newcommand{\AJ}{\mathrm{AJ}}
\newcommand{\Ad}{\mathrm{Ad}}
\newcommand{\cusp}{\mathrm{cusp}}
\newcommand{\cent}{\mathrm{Cent}}

\newcommand{\Sym}{\mathrm{Sym}}
\newcommand{\Lie}{\mathrm{Lie}}
\newcommand{\SL}{\mathrm{SL}}
\newcommand{\GL}{\mathrm{GL}}
\newcommand{\U}{\mathrm{U}}
\newcommand{\Mat}{\mathrm{Mat}}
\newcommand{\SO}{\mathrm{SO}}
\newcommand{\Sp}{\mathrm{Sp}}

\newcommand{\lam}{\lambda}
\renewcommand{\1}{{\bf 1}}
\newcommand{\I}{\sqrt{-1}}
\newcommand{\resp}{resp.~}
\newcommand{\bs}{\backslash}
\newcommand{\ep}{\varepsilon}

\newcommand{\bi}{\mathbf{i}}
\newcommand{\bk}{\mathbf{k}}

\newcommand{\pair}[1]{\langle #1 \rangle}
\newcommand{\half}[1]{\frac{#1}{2}}
\newcommand{\iif}{&\quad&\text{if }}
\newcommand{\other}{&\quad&\text{otherwise}}

\newcommand{\kk}{\mathfrak{k}}
\newcommand{\oo}{\mathfrak{o}}
\newcommand{\p}{\mathfrak{p}}
\newcommand{\g}{\mathfrak{g}}
\newcommand{\z}{\mathfrak{z}}
\renewcommand{\sp}{\mathfrak{sp}}
\newcommand{\so}{\mathfrak{so}}

\newtheorem{thm}{Theorem}[section]
\newtheorem{lem}[thm]{Lemma}
\newtheorem{prop}[thm]{Proposition}
\newtheorem{cor}[thm]{Corollary}
\newtheorem{rem}[thm]{Remark}

\newtheorem{ex}[thm]{Example}
\newtheorem{defi}[thm]{Definition}

\makeatletter
\def\iddots{\mathinner{\mkern1mu\raise\p@
	\hbox{.}\mkern2mu\raise4\p@\hbox{.}\mkern2mu
	\raise7\p@\vbox{\kern7\p@\hbox{.}}\mkern1mu}}
\def\adots{\mathinner{\mkern2mu\raise\p@\hbox{.}
 \mkern2mu\raise4\p@\hbox{.}\mkern1mu
 \raise7\p@\vbox{\kern7\p@\hbox{.}}\mkern1mu}}
\makeatother

\title{Applications of Arthur's multiplicity formula 
\\
to Siegel modular forms}
\author{Hiraku Atobe}
\date{}
\address{
Department of Mathematics, Hokkaido University
Kita 10, Nishi 8, Kita-Ku, Sapporo, Hokkaido, 060-0810, Japan 
}
\email{
atobe@math.sci.hokudai.ac.jp
}

\begin{document}
\maketitle
\begin{abstract}
We give two applications of Arthur's multiplicity formula to Siegel modular forms.
The one is a lifting theorem for vector valued Siegel modular forms, 
which contains Miyawaki's conjectures and Ibukiyama's conjectures.
The other is the strong multiplicity one theorem for Siegel modular forms of 
scalar weights and level one.
\end{abstract}

\tableofcontents

\section{Introduction}\label{intro}
In the modern number theory, modular forms and automorphic representations 
are indispensable tools.
In particular, they give automorphic $L$-functions, 
which enjoy Euler products, meromorphic continuations, and functional equations.
\vskip 10pt

We review the classical theory of Siegel modular forms.
Let $S_{2k}(\SL_2(\Z))$ (\resp $S_{k,j}(\Sp_2(\Z))$) be the space of elliptic cusp forms of weight $2k$, 
(\resp the space of Siegel modular cusp forms 
of degree $2$ and of vector weight ${\det}^k\Sym(j)$).
When $f \in S_{2k}(\SL_2(\Z))$ (\resp $f \in S_{k,j}(\Sp_2(\Z))$) is a Hecke eigenform, 
we have the Hecke $L$-function $L(s,f)$ (\resp $L(s,f, \spin)$), 
which is a degree $2$ (\resp degree $4$) $L$-function 
satisfying a functional equation with respect to $s \leftrightarrow 2k-s$ (\resp $s \leftrightarrow 1-s$).
The $L$-function $L(s,f,\spin)$ is called the spinor $L$-function associated $f$.
\par

More generally, for $\bk = (k_1, \dots, k_n) \in \Z^n$ with $k_1 \geq \dots \geq k_n$, 
we denote by $\rho_\bk$ the irreducible representation of $\U(n)$ of highest weight $\bk$, 
and by $S_{\rho_\bk}(\Sp_n(\Z))$ the space of Siegel modular cusp forms 
of degree $n$ and of vector weight $\rho_{\bk}$.
When $\bk = (k,k,\dots, k)$, then $\rho_{\bk} = {\det}^k$, 
and $S_{\rho_\bk}(\Sp_n(\Z)) = S_{k}(\Sp_n(\Z))$ is the space of Siegel modular cusp forms 
of degree $n$ and of scalar weight $k$.
In general, 
for a Hecke eigenform $F \in S_{\rho_\bk}(\Sp_n(\Z))$, 
we have an $L$-function $L(s,F,\std)$ called the standard $L$-function associated to $f$, 
which has degree $2n+1$ and satisfies a functional equation with respect to $s \leftrightarrow 1-s$.
\par

In the classical theory, there are several lifting conjectures.
We recall the following conjectures predicted by Miyawaki \cite{M}, Heim \cite{H}, and Ibukiyama \cite{Ib}.
\vskip 5pt

\noindent{\bf Miyawaki's conjecture of Type $\mathrm{I}$:}
For any Hecke eigenforms $f \in S_{2k-4}(\SL_2(\Z))$ and $g \in S_{k}(\SL_2(\Z))$, 
there should exist a Hecke eigenform $F_{f,g} \in S_{k}(\Sp_3(\Z))$ such that
\[
L(s, F_{f,g}, \std) = L(s,g,\std)L(s+k-2, f)L(s+k-3, f).
\]
\vskip 5pt

\noindent{\bf Miyawaki's conjecture of Type $\mathrm{II}$:}
For any Hecke eigenforms $f \in S_{2k-2}(\SL_2(\Z))$ and $g \in S_{k-2}(\SL_2(\Z))$, 
there should exist a Hecke eigenform $F_{f,g} \in S_{k}(\Sp_3(\Z))$ such that
\[
L(s, F_{f,g}, \std) = L(s,g,\std)L(s+k-1, f)L(s+k-2, f).
\]
\vskip 5pt

\noindent{\bf Ibukiyama's conjecture of Type $\mathrm{I}$:}
For any Hecke eigenform $f \in S_{n+2,2m-3n-2}(\Sp_2(\Z))$
with even $n,m$ such that $m > 2n$ and $n \geq 2$, 
there should exist a Hecke eigenform $F_{f} \in S_{m}(\Sp_{2n}(\Z))$ such that
\[
L(s, F_{f,g}, \std) = \zeta(s) \prod_{i=1}^n L\left( s+\half{n+1}-i, f, \spin \right).
\]
\vskip 5pt

\noindent{\bf Ibukiyama's conjecture of Type $\mathrm{II}$:}
For any Hecke eigenforms $f \in S_{2m-2n}(\SL_2(\Z))$ and $g \in S_{m-2n+2,2n-2}(\Sp_2(\Z))$ 
with $m$ even and $m > 2n-2$, 
there should exist a Hecke eigenform $F_{f,g} \in S_{m}(\Sp_{2n}(\Z))$ such that
\[
L(s, F_{f,g}, \std) = L(s,g,\std) \prod_{i=1}^{2n-2}L\left( s+m-1-i, f \right).
\]
\vskip 5pt

In fact, Miyawaki \cite{M} numerically computed the actions of Hecke operators on $F_{12}$ and $F_{14}$ 
which belong to the one dimensional vector spaces $S_{12}(\Sp_3(\Z))$ and $S_{14}(\Sp_3(\Z))$, 
respectively, 
and predicted the conjectures for them.
The general forms of Miyawaki's conjectures were given by Heim \cite{H}.
Ibukiyama \cite{Ib} also considered his lifting conjectures 
for the non-cuspidal cases in slightly wider situations.
\vskip 10pt

Nowadays, these conjectures should follow from {\bf Arthur's multiplicity formula} \cite[Theorem 1.5.2]{Ar}.
Let $\A = \A_\Q$ be the ring of adeles of $\Q$, and $\A_\fin$ be its finite part.
Arthur's multiplicity formula decomposes 
the space $\AA^2(\Sp_n(\A))$ of square-integrable automorphic forms on $\Sp_n(\Q) \bs \Sp_n(\A)$
into a direct sum of simple $\Sp_n(\A_\fin) \times (\g_\infty, K_\infty)$-modules using global $A$-packets, 
where $\g_\infty = \sp_n(\C)$ is the complexification of the Lie algebra of $\Sp_n(\R)$, 
and 
\[
K_\infty = \left\{
\begin{pmatrix}
\alpha & \beta \\ -\beta & \alpha
\end{pmatrix}
\in \Sp_n(\R)
\middle|
{}^t\alpha\alpha+{}^t\beta\beta = \1_n
\right\}
\]
is the standard maximal compact subgroup of $\Sp_n(\R)$.
Using this formula, we may try to decompose the space $S_{\rho_\bk}(\Sp_n(\Z))$ of Siegel modular forms.
However, the automorphic form $\varphi_F \in \AA^2(\Sp_n(\A))$ 
given by a Siegel modular form $F \in S_{\rho_\bk}(\Sp_n(\Z))$ 
is a ``holomorphic cusp form of vector weight $\rho_{\bk}$''.
This condition implies that the archimedean components of the automorphic representations
appearing in the representation generated $\pi_F$ by $\varphi_F$ 
are the lowest weight module $L(V_{\bk})$.
Therefore, to obtain a decomposition of $S_{\rho_{\bk}}(\Sp_n(\Z))$, we need to consider
when local $A$-packets for $\Sp_n(\R)$ contain $L(V_{\bk})$.
When $k_n > n$, this problem is solved by the works of Adams--Johnson \cite{AJ} and
Arancibia--M{\oe}glin--Renard \cite{AMR} (see Proposition \ref{discrete}). 
By this observation, we obtain a decomposition of $S_{\rho_\bk}(\Sp_n(\Z))$ in this case, 
and we can conclude the following lifting theorem.

\begin{thm}[Lifting Theorem]\label{LT}
Let $\bk = (k_1, \dots, k_n) \in \Z^n$ with $k_1 \geq \dots \geq k_n > n$, 
and let $g \in S_{\rho_\bk}(\Sp_n(\Z))$ be a Hecke eigenform.

\begin{enumerate}
\item[(A)]
For positive integers $k$ and $d$, 
we assume one of the following:
\begin{itemize}
\item
$k+d-1 < k_n-n$, $k>d$ and $k \equiv d+n \bmod 2$; or 
\item
$k-d > k_1 -1$, $k > d$ and $k \equiv d \bmod 2$.
\end{itemize}
Define $\bk' = (k_1', \dots, k_{n+2d}') \in \Z^{n+2d}$ so that $k_1' \geq \dots \geq k'_{n+2d}$ and 
\begin{align*}
&\left\{ k_1'-1, k_2'-2, \dots, k_{n+2d}'-(n+2d) \right\}
\\&=
\left\{ k_1-1, k_2-2, \dots, k_n-n \right\}
\cup
\left\{
k+d-1, k+d-2, \dots, k-d
\right\}. 
\end{align*}
Then for any Hecke eigenform $f \in S_{2k}(\SL_2(\Z))$, 
there exists a Hecke eigenform $F_{f,g} \in S_{\rho_{\bk'}}(\Sp_{n+2d}(\Z))$ such that
\[
L(s, F_{f,g}, \std) = L(s,g,\std) \prod_{i=1}^{2d}L(s+k+d-i, f).
\]

\item[(B)]
For positive integers $k$, $j$ and $d > 0$, 
we assume all of the following:
\begin{itemize}
\item
$k \equiv j \equiv 0 \bmod 2$; 
\item
$k > 2d+1$ and $j > 2d-1$; 
\item
$k_i-i \not\in [\half{j}-d+1, \half{j}+k+d-2]$ for $i = 1, \dots, n$.
\end{itemize}
Define $\bk' = (k_1', \dots, k_{n+4d}') \in \Z^{n+4d}$ so that $k_1' \geq \dots \geq k'_{n+4d}$ and 
\begin{align*}
&\left\{ k_1'-1, k_2'-2, \dots, k_{n+4d}'-(n+4d) \right\}
\\&=
\left\{ k_1-1, k_2-2, \dots, k_n-n \right\}
\\&\cup
\left\{
\half{j}+k+d-2, \half{j}+k+d-3, \dots, \half{j}+k-d-1
\right\}
\cup
\left\{
\half{j}+d, \half{j}+d-1, \dots, \half{j}-d+1
\right\}.
\end{align*}
Then for any Hecke eigenform $f \in S_{k,j}(\Sp_2(\Z))$, 
there exists a Hecke eigenform $F_{f,g} \in S_{\rho_{\bk'}}(\Sp_{n+4d}(\Z))$ such that
\[
L(s, F_{f,g}, \std) = L(s,g,\std) \prod_{i=1}^{2d}L\left( s+d+\half{1}-i, f, \spin \right).
\]
\end{enumerate}
Here, when $n=0$, we interpret $L(s,g,\std)$ to be the Riemann zeta function $\zeta(s)$.
\end{thm}

\begin{rem}
\begin{enumerate}
\item
When $n=0$ and $k > d$, $k \equiv d \bmod 2$ in Lifting Theorem (A), 
the lifting $F_{f,1} \in S_{k+d}(\Sp_{2d}(\Z))$ is 
the Duke--Imamoglu--Ibukiyama--Ikeda lift of $f \in S_{2k}(\SL_2(\Z))$ (\cite{I1}),
which is determined up to a constant multiple.

\item
When we set $(n,\bk, k, d)$ to be $(1,(k),k-2,1)$ with even $k \geq 12$ 
(\resp to be $(1,(k-2),k-1,1)$ with even $k \geq 14$) in Lifting Theorem (A), 
we obtain Miyawaki's conjecture of Type $\mathrm{I}$ (\resp of Type $\mathrm{II}$).

\item
When we set $(n,\bk, k, j, d)$ to be $(0, \emptyset, n+2, 2m-3n-2, \half{n})$ 
with even $n,m$ such that $m > 2n$ and $n \geq 2$ in Lifting Theorem (B), 
we obtain Ibukiyama's conjecture of Type $\mathrm{I}$.

\item
Ibukiyama's conjecture of Type $\mathrm{II}$ is not contained in Lifting Theorem.
However, using a fact in Remark \ref{Ib2} below, 
one can prove this conjecture.

\item
It is not known how to construct $F_{f,g}$ in general.
In \cite{I2}, Ikeda suggested a way for the construction of $F_{f,g}$ for some case in Lifting Theorem (A), 
which is called the Miyawaki lifting of $g$.
However, the non-vanishing of this lifting is unknown, i.e., 
Ikeda's construction might be identically zero.
\end{enumerate}
\end{rem}
\vskip 10pt

As another application of Arthur's multiplicity formula together with several supplementary results, 
we can get the following theorem:
\begin{thm}[Strong multiplicity one theorem]\label{SMO}
For $i = 1,2$, 
let $F_i \in S_{k_i}(\Sp_n(\Z))$ be a Hecke eigenform of scalar weight $k_i$.
Suppose that 
for almost all prime $p$, 
the Satake parameter of $F_1$ at $p$ is equal to the one of $F_2$ at $p$.
Assume further that $\{k_1, k_2\} \not= \{\half{n}, \half{n}+1\}$ if $n$ is even.
Then there exists a constant $c \in \C^\times$ such that $F_2 = c F_1$.
\end{thm}
We shall explain an outline of the proof.
\begin{description}
\item[Step 1]
Consider the $\Sp_n(\A_\fin) \times (\g_\infty, K_\infty)$-module $\pi_{F_i}$ 
generated by the cusp form $\varphi_{F_i}$ corresponding to $F_i$.
Since $F_i$ is level one, 
we see that $\pi_{F_i}$ is irreducible
(Lemma \ref{irred}).

\item[Step 2]
By the assumption, 
$\pi_{F_1}$ and $\pi_{F_2}$ are nearly equivalent to each other.
This means that $\pi_{F_1}$ belongs to the same $A$-packet as $\pi_{F_2}$
(Corollary \ref{near}).

\item[Step 3]
By the uniqueness of the unramified representation 
in a $p$-adic local $A$-packet (Proposition \ref{44}), 
and by the uniqueness of the lowest weight module 
in a real local $A$-packet (Proposition \ref{MR}), 
we see that 
$\pi_{F_1}$ and $\pi_{F_2}$ are isomorphic to each other 
as $\Sp_n(\A_\fin) \times (\g_\infty, K_\infty)$-modules.

\item[Step 4]
By the multiplicity one result in local $A$-packets (Propositions \ref{mult1}, \ref{AJ}, \ref{MR}), 
we see that 
$\pi_{F_1}$ and $\pi_{F_2}$ are equal to each other as subspaces of $\AA^2(\Sp_n(\A))$.

\item[Step 5]
By the uniqueness of the everywhere unramified lowest weight vectors in an automorphic representation, 
we see that $\varphi_{F_2}$ is a constant multiple of $\varphi_{F_1}$.
This means that $F_2 = c F_1$ for some $c \in \C^\times$.
\end{description}
\vskip 10pt

This paper is organized as follows. 
In the first part of \S \ref{sec.AMF+} (\S \ref{sec.local}--\S \ref{sec.AMF}), 
we recall the local and global $A$-packets and Arthur's multiplicity formula.
In the last part of \S \ref{sec.AMF+} (\S \ref{sec.Moe}--\S \ref{sec.AJ}), 
we review M{\oe}glin's supplementary results and Adams--Johnson packets.
In section \ref{sec.Siegel+}, we explain the theory of Siegel modular forms and holomorphic cusp forms.
Lifting Theorem (Theorem \ref{LT}) and the strong multiplicity one theorem (Theorem \ref{SMO})
are proven in \S \ref{sec.LT} and \S \ref{sec.SMO}, respectively.
In Appendix \ref{arthur-char}, 
we explain the original definition of Arthur's character formally, and compute it.

\subsection*{Acknowledgments}
This article is mainly based on author's talk 
in the conference ``21st Autumn Workshop on Number Theory''.
The author is grateful to the organizers Tamotsu Ikeda and Tomokazu Kashio 
for giving him the opportunity of this talk.
The author thanks to Tomoyoshi Ibukiyama and Hidenori Katsurada
for telling several lifting conjectures.
This work was supported by the Foundation for Research Fellowships of Japan Society 
for the Promotion of Science for Young Scientists (PD) Grant 29-193. 
\par

\section{Arthur's multiplicity formula and supplementary results}\label{sec.AMF+}
In this section, we recall the general theory for
Arthur's multiplicity formula for symplectic groups together with several supplementary results.

\subsection{Local $A$-parameters and local $A$-packets}\label{sec.local}
First, we recall the notion of local $A$-packets.
Let $\F$ be a local field of characteristic zero, 
and $L_\F$ be the Weil--Deligne group of $\F$, i.e., 
\[
L_\F = \left\{
\begin{aligned}
&W_\F \iif \text{$\F$ is archimedean}, \\
&W_\F \times \SL_2(\C) \iif \text{$\F$ is non-archimedean}, 
\end{aligned}
\right.
\]
where $W_\F$ is the Weil group of $\F$.
The symplectic group of rank $n$
is the split algebraic group over $\F$ defined by
\[
\Sp_n(R) = \left\{g \in \GL_{2n}(R) \middle| 
{}^tg 
\begin{pmatrix}
0 & -\1_n \\ \1_n & 0
\end{pmatrix}
g
=
\begin{pmatrix}
0 & -\1_n \\ \1_n & 0
\end{pmatrix}
\right\}
\]
for an $\F$-algebra $R$, 
whose Langlands dual group is $\SO_{2n+1}(\C)$.
A {\it local $A$-parameter} for $\Sp_n/\F$ is 
a homomorphism $\psi \colon L_\F \times \SL_2(\C) \rightarrow \SO_{2n+1}(\C)$
such that 
\begin{enumerate}
\item
$\psi|W_\F$ is continuous; 
\item
$\psi(W_\F)$ consists of semisimple elements; 
\item
$\psi(W_\F)$ projects onto a relatively compact subset in $\SO_{2n+1}(\C)$; 
\item
$\psi|\SL_2(\C)$ is algebraic for each $\SL_2(\C) \subset L_\F \times \SL_2(\C)$.
\end{enumerate}
Let $\Psi(\Sp_n/\F)$ be the set of conjugacy classes of local $A$-parameters for $\Sp_n/\F$.
For $\psi \in \Psi(\Sp_n/\F)$, 
one can decompose
\[
\psi = m_1 \psi_1 \oplus \dots \oplus m_r \psi_r \oplus (\psi_0 \oplus \psi_0^\vee), 
\]
where 
$\psi_1, \dots, \psi_r$ are distinct irreducible orthogonal representations of $L_\F \times \SL_2(\C)$
with multiplicities $m_1, \dots, m_r \geq 1$, 
and $\psi_0$ is a sum of irreducible representations which are not orthogonal.
Define the {\it component group $A_\psi$} of $\psi$ by 
\[
A_\psi = \bigoplus_{i=1}^r (\Z/2\Z) \alpha_{\psi_i}.
\]
Namely, $A_\psi$ is a free $\Z/2\Z$-module of rank $r$,
and $\{\alpha_{\psi_1}, \dots, \alpha_{\psi_r}\}$ is a basis of $A_\psi$
with $\alpha_{\psi_i}$ associated to $\psi_i$.
For a subrepresentation
\[
\psi' = m'_1 \psi_1 \oplus \dots \oplus m'_r \psi_r \oplus (\psi_0' \oplus \psi_0'^\vee) \subset \psi
\]
with $0 \leq m_i' \leq m_i$ and $\psi_0' \subset \psi_0$, 
we set
\[
\alpha_{\psi'} = \sum_{i=1}^r m'_i \alpha_{\psi_i} \in A_\psi.
\]
The element $\alpha_\psi = \sum_{i=1}^r m_i \alpha_{\psi_i}$
is called the central element in $A_\psi$, and is denoted by $z_\psi$.
The Pontryagin dual of $A_\psi$ is denoted by $\widehat{A_\psi}$.
\par

Let $\Irr(\Sp_n(\F))$ (\resp $\Irr_\unit(\Sp_n(\F))$) be 
the set of equivalence classes of irreducible admissible (\resp unitary) representations of $\Sp_n(\F)$.
The following is the local main theorem of Arthur's endoscopic classification.

\begin{thm}[{\cite[Theorem 1.5.1]{Ar}}]\label{151}
For each $\psi \in \Psi(\Sp_n/\F)$, 
there is a finite multiset $\Pi_\psi$ over $\Irr_\unit(\Sp_n(\F))$ with a map
\[
\Pi_\psi \rightarrow \widehat{A_\psi},\ 
\pi \mapsto \pair{\cdot, \pi}_\psi
\]
satisfying (twisted and standard) endoscopic character identities
and $\pair{z_\psi, \pi}_\psi = 1$.
Moreover, if $\F$ is non-archimedean and $\pi$ is unramified, 
then $\pair{\cdot, \pi}_\psi = \1$.
\end{thm}
We call $\Pi_\psi$ the {\it local $A$-packet associated to $\psi$}.

\subsection{Global $A$-parameters and global $A$-packets}\label{globalA}
Next, we define global $A$-packets.
Let $\F$ be a number field.
A {\it (discrete) global $A$-parameter for $\Sp_n/\F$}
is a symbol
\[
\psi = \tau_1[d_1] \boxplus \dots \boxplus \tau_r[d_r], 
\]
where
\begin{itemize}
\item
$\tau_i$ is an irreducible unitary cuspidal self-dual automorphic representation of $\GL_{m_i}(\A)$; 
\item
$d_i$ is a positive integer such that $\sum_{i=1}^r m_id_i = 2n+1$; 
\item
if $d_i$ is odd, then $\tau_i$ is orthogonal, i.e., $L(s, \tau_i, \Sym^2)$ has a pole at $s=1$; 
\item
if $d_i$ is even, then $\tau_i$ is symplectic, i.e., $L(s, \tau_i, \wedge^2)$ has a pole at $s=1$; 
\item
the central character $\omega_i$ of $\tau_i$ satisfies that $\omega_1^{d_1} \cdots \omega_r^{d_r} = \1$; 
\item
if $i \not= j$ and $\tau_i \cong \tau_j$, then $d_i \not= d_j$.
\end{itemize}
Two global $A$-parameters $\psi = \boxplus_{i=1}^{r} \tau_i[d_i]$ 
and $\psi' = \boxplus_{i=1}^{r'} \tau'_i[d'_i]$
are said to be {\it equivalent} 
if $r=r'$ and there exists a permutation $\sigma \in \mathfrak{S}_r$
such that $d'_i = d_{\sigma(i)}$ and $\tau'_i \cong \tau_{\sigma(i)}$.
We denote by $\Psi_2(\Sp_n/\F)$ be the set of equivalence classes of 
discrete global $A$-parameters for $\Sp_n/\F$.
For $\psi = \boxplus_{i=1}^r\tau_i[d_i] \in \Psi_2(\Sp_n/\F)$, 
we define the {\it component group} $A_\psi$ of $\psi$ by 
\[
A_\psi = \bigoplus_{i=1}^r (\Z/2\Z) \alpha_{\tau_i[d_i]}.
\]
Namely, $A_\psi$ is a free $\Z/2\Z$-module of rank $r$,
and $\{\alpha_{\tau_1[d_1]}, \dots, \alpha_{\tau_r[d_r]}\}$ is a basis of $A_\psi$
with $\alpha_{\tau_i[d_i]}$ associated to $\tau_i[d_i]$.
We call $z_\psi = \alpha_{\tau_1[d_1]} + \dots + \alpha_{\tau_r[d_r]} \in A_\psi$ the {\it central element}.
\par

We define {\it Arthur's character} $\ep_\psi \colon A_\psi \rightarrow \{\pm1\}$ by
\[
\ep_\psi(\alpha_{\tau_i[d_i]}) = \prod_{j \not= i}\ep(\tau_i \times \tau_j)^{\min\{d_i,d_j\}}, 
\]
where $\ep(\tau_i \times \tau_j) = \ep(1/2, \tau_i \times \tau_j) \in \{\pm1\}$ is 
the central value of the Rankin--Selberg epsilon factor $\ep(s, \tau_i \times \tau_j)$.
We note that $\ep_\psi(z_\psi) = \prod_{i=1}^r \ep_\psi(\alpha_{\tau_i[d_i]}) = 1$.
An important result of Arthur \cite[Theorem 1.5.3]{Ar} states that 
$\ep(\tau_i \times \tau_j) =1$ if $d_i \equiv d_j \bmod 2$.
In particular, if $d_1=\dots=d_r=1$, 
then $\ep_\psi$ is the trivial character of $A_\psi$.

\begin{rem}
This definition of $\ep_\psi$ might seem to be different from Arthur's original definition.
In Appendix \ref{arthur-char}, we formally explain the original definition, 
and show that it coincides with our definition (Proposition \ref{agree}).
\end{rem}
\par

Let $\psi = \boxplus_{i=1}^r\tau_i[d_i] \in \Psi_2(\Sp_n/\F)$
be a global $A$-parameter 
with $\tau_i$ being an irreducible unitary cuspidal representation of $\GL_{m_i}(\A)$.
For each place $v$ of $\F$, 
we denote by $\phi_{i,v}$ the $m_i$-dimensional representation of $L_{\F_v}$
corresponding to the irreducible representation $\tau_{i,v}$ of $\GL_{m_i}(\F_v)$.
We define a homomorphism $\psi_v \colon L_{\F_v} \times \SL_2(\C) \rightarrow \GL_{2n+1}(\C)$ by 
\[
\psi_v = (\phi_{1,v} \boxtimes S_{d_1}) \oplus \dots \oplus (\phi_{r,v} \boxtimes S_{d_r}), 
\] 
where $S_d$ is the unique irreducible algebraic representation of $\SL_2(\C)$ of dimension $d$.
We call $\psi_v$ the {\it localization of $\psi$ at $v$}.
By \cite[Proposition 1.4.2]{Ar}, 
the homomorphism $\psi_v$ factors through $\SO_{2n+1}(\C) \hookrightarrow \GL_{2n+1}(\C)$.
However, the localization $\psi_v$ is not necessarily a local $A$-parameter for $\Sp_n/\F_v$
because of the lack of the generalized Ramanujan conjecture.
One can define the component group $A_{\psi_v}$ of $\psi_v$ similar to local $A$-parameters.
There is a localization map
\[
A_\psi \rightarrow A_{\psi_v},\ 
\alpha_{\tau_i[d_i]} \mapsto \alpha_{\phi_{i,v} \boxtimes S_{d_i}}.
\]
In particular, we obtain the diagonal map
\[
\Delta \colon A_\psi \rightarrow \prod_v A_{\psi_v}.
\]
\par

Let $\psi \in \Psi_2(\Sp_n/\F)$ be a global $A$-parameter 
and $\psi_v$ be the localization of $\psi$ at $v$.
We can decompose 
\[
\psi_v = \psi_1|\cdot|^{s_1}_{\F_v} \oplus \dots \oplus \psi_r|\cdot|^{s_r}_{\F_v} 
\oplus \psi_0 \oplus \psi_r^\vee|\cdot|^{-s_r}_{\F_v} \oplus \dots \oplus \psi_1^\vee|\cdot|^{-s_1}_{\F_v}, 
\]
where
\begin{itemize}
\item
$\psi_i$ is an irreducible representation of $L_{\F_v} \times \SL_2(\C)$ of dimension $d_i$
such that $\psi_i(W_{\F_v})$ is bounded for $i = 1, \dots, r$; 
\item
$\psi_0 \in \Psi(\Sp_{n_0}/\F_v)$; 
\item
$d_1+ \dots + d_r + n_0 = n$ and $s_1 \geq \dots \geq s_r > 0$.
\end{itemize}
We note that $s_i < 1/2$ by the toward Ramanujan conjecture 
(see \cite[(2.5) Corollary]{JS1} and \cite[Appendix]{RS}).
We define a representation $\phi_{\psi_i}$ of $L_{\F_v}$ by 
\[
\phi_{\psi_i}(w) = \psi_i
\left(w, 
\begin{pmatrix}
|w|_{\F_v}^{\half{1}} & 0 \\ 0 & |w|_{\F_v}^{-\half{1}}
\end{pmatrix}
\right), 
\quad 
w \in L_{\F_v}, 
\]
and we denote by $\tau_{\psi_i}$ the irreducible representation of $\GL_{d_i}(\F_v)$
corresponding to $\phi_{\psi_i}$.
Let $\Pi_{\psi_0}$ be the local $A$-packet associated to $\psi_0$, 
which is a multiset over $\Irr_\unit(\Sp_{n_0}(\F_v))$.
For $\pi_0 \in \Pi_{\psi_0}$, we put
\[
I_{\psi_v}(\pi_0) = \Ind_{P(\F_v)}^{\Sp_n(\F_v)}
(\tau_{\psi_1}|\cdot|^{s_1}_{\F_v} \boxtimes \dots \boxtimes \tau_{\psi_r}|\cdot|^{s_r}_{\F_v} \boxtimes \pi_0), 
\]
where $P$ is a parabolic subgroup of $\Sp_n$ with Levi subgroup
$M_P = \GL_{d_1} \times \dots \times \GL_{d_r} \times \Sp_{n_0}$.
The local $A$-packet $\Pi_{\psi_v}$ associated to the localization $\psi_v$, 
which is a multiset over $\Irr(\Sp_n(\F_v))$, 
is defined by the disjoint union of the multisets of 
the Jordan--H{\"o}lder series of $I_{\psi_v}(\pi_0)$, 
i.e., 
\[
\Pi_{\psi_v} = \bigsqcup_{\pi_0 \in \Pi_{\psi_0}}\Big\{\pi \ \Big| \ 
\text{$\pi$ is an irreducible constituent of $I_{\psi_v}(\pi_0)$}
\Big\}.
\]
Note that $A_{\psi_v} = A_{\psi_0}$.
Define a map 
\[
\Pi_{\psi_v} \rightarrow \widehat{A_{\psi_v}},\ 
\pi \mapsto \pair{\cdot, \pi}_{\psi_v}
\]
by
\[
\pair{\cdot, \pi}_{\psi_v} = \pair{\cdot, \pi_0}_{\psi_0}
\]
if $\pi$ is an irreducible constituent of $I_{\psi_v}(\pi_0)$.
\par

When $v \mid \infty$, we fix a maximal compact subgroup $K_v$ of $\Sp_n(\F_v)$, 
and set $K_\infty = \prod_{v \mid \infty} K_v$.
When $v$ is a real place, we choose $K_v$ as
\[
K_v = \left\{
\begin{pmatrix}
\alpha & \beta \\ -\beta & \alpha
\end{pmatrix}
\in \Sp_n(\F_v) \cong \Sp_n(\R)
\middle|
{}^t\alpha\alpha+{}^t\beta\beta = \1_n
\right\}.
\]
For an irreducible admissible representation $\pi_\infty = \otimes_{v \mid \infty}\pi_v$ 
of $\Sp_n(\F \otimes_\Q \R) \cong \prod_{v \mid \infty}\Sp_n(\F_v)$, 
we denote the $K_\infty$-finite part of $\pi_\infty$ 
by $\pi_{\infty, K_\infty} = \otimes_{v \mid \infty}\pi_{v,K_v}$.
This is a simple $(\g_\infty, K_\infty)$-module, 
where $\g_\infty = \Lie(\Sp_n(\F \otimes_\Q \R)) \otimes_\R \C$
is the complexification of the Lie algebra of $\Sp_n(\F \otimes_\Q \R)$.
\par

Now we are ready to define global $A$-packets.
For a global $A$-parameter $\psi \in \Psi_2(\Sp_n/\F)$, 
we define the global $A$-packet $\Pi_\psi$ by
\[
\Pi_\psi = \left\{
\pi = \left({\bigotimes_{v < \infty}}' \pi_v\right) \otimes \left( \bigotimes_{v \mid \infty} \pi_{v,K_{v}}\right)
\ \middle|\ 
\pi_v \in \Pi_{\psi_v},\ 
\text{$\pair{\cdot, \pi_v}_{\psi_v} = \1$ for almost all $v$}
\right\}.
\]
This is a multiset over the set of 
equivalence classes of simple admissible $\Sp_n(\A_\fin) \times (\g_\infty, K_\infty)$-modules.
By abuse of notation, we simply write $\pi = \otimes'_v \pi_v$ 
for $\pi = (\otimes'_{v<\infty} \pi_v) \otimes (\otimes_{v \mid \infty}\pi_{v,K_v})$. 
For $\pi = \otimes'_v \pi_v \in \Pi_\psi$, 
since $\pair{\cdot, \pi_v}_{\psi_v} = \1$ for almost all $v$, 
one can define a character $\pair{\cdot, \pi}_\psi$ of $\prod_vA_{\psi_v}$ by 
\[
\pair{\cdot, \pi}_\psi = \bigotimes_v \pair{\cdot, \pi_v}_{\psi_v}.
\]

\subsection{Arthur's multiplicity formula}\label{sec.AMF}
We state Arthur's multiplicity formula.
Let $\F$ be a number field.
For a finite place $v < \infty$, 
we denote by $\oo_v$ the ring of integers of $\F_v$.
A smooth function
\[
\varphi \colon \Sp_n(\A) \rightarrow \C
\]
is a {\it square-integrable automorphic form on $\Sp_n(\A)$}
if $\varphi$ satisfies the following conditions:
\begin{enumerate}
\item
$\varphi$ is left $\Sp_n(\F)$-invariant; 
\item
$\varphi$ is right $K$-invariant, 
where $K = K_\fin \times K_\infty$ is the maximal compact subgroup of $\Sp_n(\A)$
with $K_\fin = \prod_{v < \infty}\Sp_n(\oo_v)$; 
\item
$\varphi$ is $\z$-finite, 
where $\z$ is the center of the universal enveloping algebra $\UU(\g_\infty)$ of $\g_\infty$.
\item
$\varphi$ is square-integrable, i.e., 
\[
\int_{\Sp_n(\F) \bs \Sp_n(\A)} |\varphi(g)|^2 dg < \infty.
\]
\end{enumerate}
By \cite[\S 4.3]{BJ}, such a function is of moderate growth.
We also note that any cusp form on $\Sp_n(\A)$ is square-integrable.
The set of square-integrable automorphic forms on $\Sp_n(\A)$ is denoted by $\AA^2(\Sp_n(\A))$.
It is an $\Sp_n(\A_\fin) \times (\g_\infty, K_\infty)$-module.
\par

Recall that for an $A$-parameter $\psi \in \Psi_2(\Sp_n/\F)$, 
\begin{itemize}
\item
its component group $A_\psi$ has Arthur's character $\ep_\psi$; 
\item
there is a diagonal map $\Delta \colon A_\psi \rightarrow \prod_v A_{\psi_v}$; 
\item
we have a character $\pair{\cdot, \pi}_\psi$ on $\prod_v A_{\psi_v}$
associated to $\pi \in \Pi_\psi$.
\end{itemize}
Arthur's multiplicity formula (\cite[Theorem 1.5.2]{Ar}) gives a decomposition of $\AA^2(\Sp_n(\A))$.
\begin{thm}[Arthur's multiplicity formula]\label{AMF}
The discrete spectrum $\AA^2(\Sp_n(\A))$ decomposes into a direct sum
\[
\AA^2(\Sp_n(\A)) \cong \bigoplus_{\psi \in \Psi_2(\Sp_n/\F)} \bigoplus_{\pi \in \Pi_\psi} m_{\pi, \psi} \pi
\]
as an $\Sp_n(\A_\fin) \times (\g_\infty, K_\infty)$-module, 
where the non-negative integer $m_{\pi, \psi}$ is given by
\[
m_{\pi, \psi} = \left\{
\begin{aligned}
&1 \iif \pair{\cdot, \pi}_\psi \circ \Delta = \ep_\psi, \\
&0 \other.
\end{aligned}
\right.
\]
\end{thm}

We emphasize that 
Arthur's multiplicity formula does not tell us whether $\AA^2(\Sp_n(\A))$ is multiplicity-free or not
since the $A$-packets are multisets.
In order to investigate $\AA^2(\Sp_n(\A))$ more precisely, 
we need to study the structures of local $A$-packets $\Pi_{\psi_v}$.

\subsection{M{\oe}glin's results}\label{sec.Moe}
In the successive works \cite{Moe1, Moe2, Moe3, Moe4, Moe5},  
M{\oe}glin has defined local $A$-packets herself for $p$-adic fields, 
and given many important results.
Xu \cite{X} proved that M{\oe}glin's $A$-packets coincide with Arthur's ones.
Hence one can use M{\oe}glin's results for Arthur's $A$-packets.
\par

First, we let $\F$ be a non-archimedean local field of characteristic zero, 
with the the ring of integers $\oo$.
We denote the cardinality of the residual field of $\F$ by $q$.
\par

For a local $A$-parameter $\psi \in \Psi(\Sp_n/\F)$, 
we have an $A$-packet $\Pi_\psi$, 
which is a multiset over $\Irr_\unit(\Sp_n(\F))$.
One of the most important result of M{\oe}glin is as follows:

\begin{thm}[{\cite{Moe4}, \cite[Theorem 8.12]{X}}]\label{mult1}
The $A$-packet $\Pi_{\psi}$ is multiplicity-free, i.e., 
it is in fact a subset of $\Irr_\unit(\Sp_n(\F))$.
\end{thm}
\par

Recall that an $A$-parameter is a homomorphism 
$\psi \colon W_\F \times \SL_2(\C) \times \SL_2(\C) \rightarrow \SO_{2n+1}(\C)$.
Let $\Delta \colon W_\F \times \SL_2(\C) \rightarrow W_\F \times \SL_2(\C) \times \SL_2(\C)$
be the diagonal map given by $\Delta(w,a) = (w,a,a)$.
For two $A$-parameters $\psi, \psi' \in \Psi(\Sp_n/\F)$, 
their $A$-packets $\Pi_\psi$ and $\Pi_{\psi'}$ can have intersection.

\begin{prop}[{\cite[Corollaire 4.2]{Moe3}}]\label{42}
If $\Pi_{\psi} \cap \Pi_{\psi'} \not= \emptyset$, 
then the diagonal restrictions $\psi \circ \Delta$ and $\psi' \circ \Delta$ are conjugate by $\SO_{2n+1}(\C)$.
\end{prop}

Recall that an irreducible representation $\pi$ of $\Sp_n(\F)$ is unramified 
if $\pi$ has a nonzero $\Sp_n(\oo)$-fixed vector.
In this case, $\pi$ is uniquely determined by its Satake parameter $c(\pi)$, 
which is a semisimple conjugacy class of $\SO_{2n+1}(\C)$.
We say that 
an $A$-parameter $\psi \colon L_\F \times \SL_2(\C) \rightarrow \SO_{2n+1}(\C)$ for $\Sp_n/\F$ 
is {\it unramified}
if the restriction of $\psi$ to $L_\F = W_\F \times \SL_2(\C)$ factors through 
the quotient $W_\F \times \SL_2(\C) \twoheadrightarrow W_\F/I_\F$, 
where $I_\F$ is the inertia subgroup of $W_\F$.
In this case, one can write 
\[
\psi = |\cdot|_\F^{s_1} \boxtimes S_{d_1} \oplus \dots \oplus |\cdot|_\F^{s_r} \boxtimes S_{d_r}
\]
with $s_i \in \I\R$.

\begin{prop}[{\cite[Proposition 4.4]{Moe3}}]\label{44}
When $\psi$ is unramified and is of the above form, 
the $A$-packet $\Pi_\psi$ has a unique unramified representation $\pi$.
Its Satake parameter $c(\pi)$ is given by
\[
c(\pi) = \bigoplus_{i=1}^r 
\begin{pmatrix}
q^{-s_i+\half{d_i-1}} &&&\\
&q^{-s_i+\half{d_i-3}}&&\\
&&\ddots&\\
&&&q^{-s_i-\half{d_i-1}}
\end{pmatrix}.
\]
By Theorem \ref{151}, 
the character $\pair{\cdot, \pi}_\psi$ of $A_\psi$ is trivial.
\end{prop}
\par

Next, we let $\F$ be a number field and $v$ be a finite place of $\F$.
Recall that
for a global $A$-parameter $\psi \in \Psi_2(\Sp_n/\F)$, 
its localization $\psi_v$ at $v$ might not be a local $A$-parameter.
The local $A$-packet $\Pi_{\psi_v}$ is the disjoint union of the Jordan--H{\"o}lder series of 
induced representations $I_{\psi_v}(\pi_0)$, 
where $\pi_0$ runs over the $A$-packet associated to a local $A$-parameter $\psi_0 \in \Psi(\Sp_{n_0}/\F)$.

\begin{prop}[{\cite[Proposition 5.1]{Moe5}}]\label{51}
For $\pi_0 \in \Pi_{\psi_0}$, 
the induced representation $I_{\psi_v}(\pi_0)$ is irreducible.
Moreover, if $\pi_0, \pi_0' \in \Pi_{\psi_0}$ are not isomorphic to each other, 
then $I_{\psi_v}(\pi_0) \not\cong I_{\psi_v}(\pi_0')$.
In conclusion, $\Pi_{\psi_v}$ is a (multiplicity-free) subset of $\Irr(\Sp_n(\F))$.
\end{prop}

For two irreducible admissible representations
$\pi \cong \otimes_v' \pi_v$ and $\pi' \cong \otimes'_v \pi_v'$ of $\Sp_n(\A)$, 
we say that $\pi$ is {\it nearly equivalent to $\pi'$} if $\pi_v \cong \pi_v'$ for almost all $v$.

\begin{cor}\label{near}
Let $\psi, \psi' \in \Pi_\psi$ be two global $A$-parameters, 
and take $\pi \in \Pi_\psi$ and $\pi' \in \Pi_{\psi'}$.
Then $\pi$ is nearly equivalent to $\pi'$ if and only if $\psi$ is equivalent to $\psi'$.
\end{cor}
\begin{proof}
Suppose that $\pi, \pi' \in \Pi_\psi$.
Since $\pi_v$ and $\pi'_v$ are unramified for almost all $v$, 
by Propositions \ref{44} and \ref{51}, 
we have $\pi_v \cong \pi_v'$ for almost all $v$.
\par

Conversely, suppose that $\pi$ is nearly equivalent to $\pi'$.
Let $v$ be a finite place of $\F$ such that all of $\pi_v$, $\pi_v'$, $\psi_v$ and $\psi_v'$ are unramified, 
and $\pi_v \cong \pi_v'$.
Write $\pi_v = I_{\psi_v}(\pi_0)$ and $\pi'_v = I_{\psi_v'}(\pi'_0)$
with $\pi_0 \in \Pi_{\psi_0}$ and $\pi_0' \in \Pi_{\psi_0'}$ 
for $\psi_0 \in \Psi(\Sp_{n_0}/\F_v)$ and $\psi_0' \in \Psi(\Sp_{n_0'}/\F_v)$.
Comparing the real parts of the exponents of the eigenvalues of the Satake parameters of $\pi$ and $\pi'$, 
we see that $n_0 = n_0'$ and $\pi_0 \cong \pi_0'$.
By Proposition \ref{42}, we have $\psi_0 \circ \Delta \cong \psi_0' \circ \Delta$.
This implies that $\psi_0 \cong \psi_0'$ since they are both unramified.
Comparing the Satake parameters of $\pi$ and $\pi'$ again, 
we see that $\psi_v \cong \psi_v'$.
As in \cite[\S 1.3]{Ar}, 
global $A$-parameters $\psi$ and $\psi'$ 
define isobaric sums of representations $\tau_{\psi}$ and $\tau_{\psi'}$, 
which are automorphic representations of $\GL_{2n+1}(\A)$.
By the generalized strong multiplicity one theorem of Jacquet--Shalika \cite[Theorem (4.4)]{JS2}, 
the condition $\psi_v \cong \psi_v'$ for almost all $v$ implies that  $\tau_\psi \cong \tau_{\psi'}$.
Since $\psi \mapsto \tau_\psi$ is injective, 
we conclude that $\psi \cong \psi'$.
\end{proof}

\subsection{Adams--Johnson packets}\label{sec.AJ}
In this subsection, we consider $\F = \R$.
In this case, the theory of local $A$-packets is not satisfactory now.
On the other hand, 
Adams and Johnson \cite{AJ} constructed $A$-packets for a certain special class of $A$-parameters.
These $A$-packets are sets and their elements are given by cohomological induction
so that these $A$-packets are relatively easy to understand.
Recently, Arancibia, M{\oe}glin and Renard \cite{AMR} proved that 
the $A$-packets of Adams--Johnson coincide with Arthur's ones.
Hence one can easily understand Arthur's $A$-packets for certain spacial parameters.
\par

Recall that the Weil groups of $\C$ and $\R$ are given by
\[
W_\C = \C^\times,
\quad
W_\R = \C^\times \sqcup \C^\times j, 
\]
respectively, 
where 
\[
j^2 = -1, 
\quad
jzj^{-1} = \overline{z}
\quad
\text{for $z \in \C^\times$}.
\]
There exists a canonical exact sequence
\[
\begin{CD}
1 @>>> W_\C @>>> W_\R @>>> \Gal(\C/\R) @>>> 1.
\end{CD}
\]
\par

There are exactly two quadratic characters of $W_\R$.
One is the trivial representation $\1$, and the other is the sign character
\[
\sgn \colon W_\R \rightarrow \{\pm1\}
\]
given by $\sgn(j) = -1$ and $\sgn(z) = 1$ for $z \in \C^\times$.
\par

For each integer $\alpha$, 
we define a $2$-dimensional representation 
\[
\rho_\alpha \colon W_\R \rightarrow \GL_2(\C)
\]
by 
\[
\rho_\alpha(j) = \begin{pmatrix}
0 & (-1)^\alpha \\ 1 & 0
\end{pmatrix},
\quad
\rho_\alpha(z) = \begin{pmatrix}
\chi_{\alpha}(z) & 0 \\ 0 & \chi_{-\alpha}(z)
\end{pmatrix}
\quad
\text{for $z \in \C^\times$}, 
\]
where the character $\chi_{\alpha}$ of $\C^\times$ is defined by 
$\chi_\alpha(z) = \overline{z}^{-\alpha}(z\overline{z})^{\half{\alpha}}$.
We also write $\chi_{\alpha}(z) = (z/\overline{z})^{\half{\alpha}}$. 
Then we see that:
\begin{itemize}
\item
$\rho_\alpha$ is irreducible when $\alpha \not= 0$; 
\item
$\rho_0 \cong \1 \oplus \sgn$; 
\item
$\rho_\alpha \cong \rho_{-\alpha}$; 
\item
$\rho_\alpha$ is orthogonal (\resp symplectic) if $\alpha$ is even (\resp $\alpha$ is odd).
\end{itemize}
For $\alpha \geq \beta > 0$, 
the root number $\ep(\rho_\alpha \otimes \rho_\beta)$ is given by
\[
\ep(\rho_\alpha \otimes \rho_\beta) = \left\{
\begin{aligned}
&-1 \iif \alpha \equiv 0,\ \beta \equiv 1 \bmod 2, \\
&1 \other.
\end{aligned}
\right.
\]
Moreover, for $\alpha > 0$ and $\delta \in \{0,1\}$,
we have $\ep(\rho_\alpha \otimes \sgn^\delta) = \I^{\alpha+1}$.
\par

For an $A$-parameter $\psi \colon W_\R \times \SL_2(\C) \rightarrow \SO_{2n+1}(\C)$, 
define $\psi_d \colon W_\C \rightarrow \SO_{2n+1}(\C)$ by 
\[
\psi_d(z) = \psi
\left(z, 
\begin{pmatrix}
(z/\overline{z})^{\half{1}} & 0 \\ 0 & (z/\overline{z})^{-\half{1}}
\end{pmatrix}
\right)
\quad 
\text{for $z \in \C^\times$}.
\]
We call an $A$-parameter $\psi \in \Psi(\Sp_n/\R)$ {\it Adams--Johnson}
if $\psi$ is a direct sum of irreducible orthogonal representations of $W_\R \times \SL_2(\C)$, 
and $\psi_d$ is multiplicity-free.
Let $\Psi_\AJ(\Sp_n/\R)$ be the subset of $\Psi(\Sp_n/\R)$ consisting of Adams--Johnson $A$-parameters.
It is easy to see that $\psi$ is Adams--Johnson if and only if $\psi$ is of the form
\[
\psi = \left(\bigoplus_{i=1}^r \rho_{\alpha_i} \boxtimes S_{d_i}\right) \oplus \sgn^\delta \boxtimes S_{d_0}, 
\]
where 
\begin{itemize}
\item
$\alpha_i > 0$ and $d_i > 0$ for $1 \leq i \leq r$; 
\item
$2\sum_{i=1}^r d_i + d_0 = 2n+1$; 
\item
$\alpha_i + d_i \equiv 1 \bmod 2$ for $1 \leq i \leq r$, and $d_0 \equiv 1 \bmod 2$; 
\item
$\delta \in \{0,1\}$ such that $\delta \equiv \sum_{i=1}^r d_i \bmod 2$; 
\item
$\alpha_i-\alpha_{i+1} \geq d_i+d_{i+1}$ for $1 \leq i < r$, 
and $\alpha_r \geq d_r+d_0$. 
\end{itemize}
In this subsection, we fix such $\psi$. 
\par

We denote the standard Cartan involution $g \mapsto {}^tg^{-1}$ of $\Sp_n(\R)$ by $\theta$.
Let
\[
T(\R) = \left\{
\left(
\begin{array}{ccc|ccc}
a_1&&&b_1&&\\
&\ddots&&&\ddots&\\
&&a_n&&&b_n\\
\hline
-b_1&&&a_1&&\\
&\ddots&&&\ddots&\\
&&-b_n&&&a_n
\end{array}
\right)
\in \Sp_n(\R)
\ \middle|\ 
a_i^2+b_i^2=1
\right\}
\]
be a maximal torus of $\Sp_n(\R)$, which is compact.
Fix a (standard) Borel subgroup $B$ of $\Sp_n$ containing $T$
(c.f., see \cite[\S 2.2]{MR}).
We write $L(\R)$ for the subgroup of $\Sp_n(\R)$ consisting of elements of the form
\[
\left(
\begin{array}{ccc|c|ccc|c}
a_1 &&& &b_1&&&\\
&\ddots&& &&\ddots&&\\
&&a_r& &&&b_r&\\
\hline
&&&A &&&&B \\
\hline
-b_1 &&& &a_1&&&\\
&\ddots&& &&\ddots&&\\
&&-b_r& &&&a_r&\\
\hline
&&&C &&&&D 
\end{array}
\right)
\]
where $a_i + \I b_i \in \U(d_i)$ for $i = 1, \dots, r$ and 
\[
\begin{pmatrix}
A & B \\ C & D
\end{pmatrix}
\in \Sp_{d_0-1}(\R).
\]
Hence 
\[
L(\R) \cong \U(d_1) \times \dots \times \U(d_r) \times \Sp_{d_0-1}(\R).
\]
Set 
\[
\Sigma_\psi = W(L(\C),T(\C)) \bs W(\Sp_n(\C),T(\C)) / W(\Sp_n(\R),T(\R)).
\]
For a positive integer $d$, 
we define $\PP_2(d)$ by the set of pairs of non-negative integers $(p,q)$ such that $p+q = d$.
Then we have a canonical bijection
\[
\Sigma_\psi \cong \prod_{i=1}^r \PP_2(d_i).
\]
For $\{(p_i,q_i)\}_i \in \prod_{i=1}^r \PP_2(d_i)$, 
if we set
\[
t(p_i, q_i) = \mathrm{diag}(\underbrace{1,\dots,1}_{p_i}, 
\underbrace{\I, \dots, \I}_{q_i}) \in \GL_{d_i}(\C), 
\]
then a representative of the element in $\Sigma_\psi$ 
corresponding to $\{(p_i,q_i)\}_i \in \prod_{i=1}^r \PP_2(d_i)$ 
is given by
\[
t(\{(p_i,q_i)\}_i) = 
\left(
\begin{array}{ccc|c|ccc|c}
t(p_i,q_i) &&&&&&&\\
&\ddots&&&&&&\\
&&t(p_r,q_r)&&&&&\\
\hline
&&&\1_{d_0-1}&&&&\\
\hline
&&&&t(p_1,q_1)^{-1}&&&\\
&&&&&\ddots&&\\
&&&&&&t(p_r,q_r)^{-1}&\\
\hline
&&&&&&&\1_{d_0-1}
\end{array}
\right).
\]
It is easy to see that 
$L_{\{(p_i,q_i)\}} = t(\{(p_i,q_i)\}_i) \cdot L \cdot t(\{(p_i,q_i)\}_i)^{-1}$ is defined over $\R$, and 
\[
L_{\{(p_i,q_i)\}}(\R) 
\cong 
\U(p_1,q_1) \times \dots \times \U(p_r,q_r) \times \Sp_{d_0-1}(\R).
\]
\par

Set 
\[
\lam_j = \left( \half{\alpha_j+d_j-1}, \half{\alpha_j+d_j-3}, \dots \half{\alpha_j-d_j+1} \right) 
\in \Z^{d_j}
\]
and 
\[
\lam_\psi = 
\left( \lam_1, \dots, \lam_r, \half{d_0-1}, \half{d_0-3}, \dots, 1 \right) \in \Z^{n}.
\]
Let $\rho = (n,n-1, \dots, 1)$ be the half sum of the positive roots of $T$ with respect to $B$. 
Then there exists a unitary character 
$\chi_{\lam_\psi} \colon L(\R) \rightarrow \C^\times$ such that
the restriction to $T(\R)$ is $\lam_\psi-\rho \in \Z^n \cong X^*(T)$. 
\par

For $w \in W(\Sp_n(\C), T(\C))$, 
we define
\[
\pi_w = A_{w^{-1}Qw}(w^{-1}\chi_{\lam_\psi})
\]
to be the derived functor module, 
where $Q$ is a $\theta$-stable parabolic subgroup of $\Sp_n$ with Levi subgroup $L$.
It is nonzero and irreducible with infinitesimal character $\lam_\psi$.
Moreover, $\pi_{w} \cong \pi_{w'}$ if and only if the images of $w$ and $w'$ in $\Sigma_\psi$
are equal to each other.

\begin{thm}[\cite{AJ}, \cite{AMR}]\label{AJ}
For $\psi = \oplus_{i=1}^r \rho_{\alpha_i} \boxtimes S_{d_i} \oplus \sgn^\delta \boxtimes S_{d_0} 
\in \Psi_\AJ(\Sp_n/\R)$ with $\alpha_i-\alpha_{i+1} \geq d_i + d_{i+1}$ for $1 \leq i < r$.
Then the $A$-packet $\Pi_\psi$ is given by the (multiplicity-free) set
\[
\Pi_\psi = \{\pi_w \ |\ w \in \Sigma_{\psi}\}.
\]
When $w \in \Sigma_\psi$ corresponds to $\{(p_i,q_i)\}_i \in \prod_{i=1}^r \PP_2(d_i)$, 
the character $\pair{\cdot, \pi_w}_\psi$ is given so that
$\pair{z_\psi, \pi_w}_\psi = 1$ and 
\[
\pair{\alpha_{\rho_{\alpha_i} \boxtimes S_{d_i}}, \pi_w}_\psi = 
(-1)^{\half{p_i-q_i-\delta_i}} 
\]
for $i = 1, \dots, r$, 
where 
\[
\delta_{i} = \left\{
\begin{aligned}
&0 \iif \text{$d_i$ is even}, \\
&(-1)^{\sum_{j=1}^{i-1}d_j} \iif \text{$d_i$ is odd}.
\end{aligned}
\right.
\]
\end{thm}

For example, suppose that $d_1 = \dots = d_r = d_0 = 1$.
We note $r = n$ and $\delta \equiv n \bmod 2$.
If $w \in \Sigma_\psi$ corresponds to $\{(p_i,q_i)\}_i \in \PP_2(1)^n$, 
then $\pi_w$ is the discrete series representation of $\Sp_n(\R)$
with Harish-Chandra parameter 
$((p_1-q_1)\lam_1, \dots, (p_n-q_n)\lam_n)$, 
where the infinitesimal character of $\pi_w$ is $\lam_\psi = (\lam_1, \dots, \lam_n)$.
The character $\pair{\cdot, \pi_w}_\psi$ satisfies that
\[
\pair{\alpha_{\rho_{\alpha_i}}, \pi_w}_\psi = (-1)^{i-1}(p_i-q_i)
\]
for $i=1, \dots, n$.

\section{Siegel modular forms and holomorphic cusp forms}\label{sec.Siegel+}
In this section, we apply Arthur's multiplicity formula to Siegel modular forms.
We prove two theorems in Introduction (\S \ref{intro}).

\subsection{Siegel modular forms}\label{sec.Siegel}
In this subsection, we put $\F = \Q$.
Let
\[
\Ha_n = \left\{Z \in \Mat_n(\C)\ |\ {}^tZ = Z,\ \im(Z) > 0\right\}
\]
be the Siegel upper half space of degree $n$.
Here, for a real symmetric matrix $Y$, we write $Y>0$ if $Y$ is positive definite.
The symplectic group $\Sp_n(\R)$ acts on $\Ha_n$ by 
\[
g\pair{Z} = (AZ+B)(CZ+D)^{-1},
\quad
g = \begin{pmatrix}
A & B \\ C & D
\end{pmatrix}
\in \Sp_n(\R), \ 
Z \in \Ha_n.
\]
The canonical automorphy factor $J(g,Z)$ is defined by
\[
J(g,Z) = CZ+D \in \GL_n(\R),
\quad
g = \begin{pmatrix}
A & B \\ C & D
\end{pmatrix}
\in \Sp_n(\R), \ 
Z \in \Ha_n.
\]
For $\bk = (k_1, \dots, k_n) \in \Z^n$ with $k_1 \geq \dots \geq k_n$, 
the irreducible representation of $\U(n)$ with highest weight $(k_1, \dots, k_n)$
is denoted by $(\rho_\bk, V_{\bk})$.
It is extended to a holomorphic representation of $\GL_n(\R)$.
Then $\rho_\bk(J(g, Z)) \in \GL(V_\bk)$ is an automorphy factor.
When $\bk = (k, k, \dots, k)$, we have $(\rho_\bk, V_{\bk}) = ({\det}^k, \C)$.

\begin{defi}
Let $\bk = (k_1, \dots, k_n) \in \Z^n$ with $k_1 \geq \dots \geq k_n$. 
A $V_\bk$-valued holomorphic function 
$F \colon \Ha_n \rightarrow V_{\bk}$
is a Siegel modular cusp form of vector weight $\rho_\bk$ 
if 
\begin{enumerate}
\item
$F(\gamma \pair{Z}) = \rho_{\bk}(J(\gamma, Z)) F(Z)$ for $\gamma \in \Sp_n(\Z)$ and $Z \in \Ha_n$;
\item
$F$ has a Fourier expansion of the form
\[
F(Z) = \sum_{\substack{T \in \Sym_n(\Q)\\ T > 0}}A_F(T) e^{2\pi\I \tr(TZ)}, 
\quad A_F(T) \in V_\bk.
\]
\end{enumerate}
\end{defi}

The space of Siegel modular cusp forms of vector weight $\rho_\bk$
is denoted by $S_{\rho_\bk}(\Sp_n(\Z))$.
There is a Hecke theory for $S_{\rho_\bk}(\Sp_n(\Z))$.
When $F \in S_{\rho_\bk}(\Sp_n(\Z))$ is a Hecke eigenform, 
for each prime $p$, 
the {\it Satake parameter} 
\[
(\beta_{p,1}^{\pm}, \dots, \beta_{p,n}^{\pm}) \in (\C^\times)^n/\mathfrak{S_n} \ltimes \{\pm1\}^n
\]
is associated to $F$.
Then the {\it standard $L$-function attached to $F$} is defined by 
\[
L(s,F,\std) = \prod_p 
\left(
(1-p^{-s})^{-1}\prod_{i=1}^n (1-\beta_{p,i}p^{-s})^{-1}(1-\beta_{p,i}^{-1}p^{-s})^{-1}
\right)
\]
for $\re(s) \gg 0$.
\par

Set $\bi = \I \cdot \1_n \in \Ha_n$.
The stabilizer of $\bi$ in $\Sp_n(\R)$ is the standard maximal compact subgroup
\[
K_\infty = \left\{
\begin{pmatrix}
\alpha & \beta \\ -\beta & \alpha
\end{pmatrix}
\in \Sp_n(\R)
\middle|
{}^t\alpha\alpha+{}^t\beta\beta = \1_n
\right\}.
\]
For a Siegel modular cusp form $F \in S_{\rho_\bk}(\Sp_n(\Z))$, 
one can define a $V_{\rho_\bk}$-valued function
$\Phi_F \colon \Sp_n(\F) \bs \Sp_n(\A) \rightarrow V_{\rho_\bk}$ by 
\[
\Phi_F (\gamma g_\infty \kappa_\fin) = \rho_{\bk}(J(g_\infty, \bi))^{-1} F(g_\infty\pair{\bi})
\]
for $\gamma \in \Sp_n(\Q)$, $g_\infty \in \Sp_n(\R)$ and $\kappa_\fin \in K_\fin = \Sp_n(\widehat{\Z})$.
By a similar argument to \cite[\S 4, Lemmas 5, 7]{AS}, $\Phi_F$ satisfies the following:
\begin{enumerate}
\item
$\Phi_F(g \kappa_\infty) = \overline{\rho_{\bk}}(\kappa_\infty)^{-1} \Phi_F(g)$
for $\kappa_\infty \in K_\infty \subset \Sp_n(\R)$, 
where $K_\infty$ is identified with $\U(n)$ by 
\[
\begin{pmatrix}
\alpha & \beta \\ -\beta & \alpha
\end{pmatrix}
\mapsto 
\alpha+\beta\I, 
\]
and the representation $\overline{\rho_{\bk}}$ of $\U(n)$ is defined by 
$\overline{\rho_{\bk}}(a) =  \rho_\bk(\overline{a})$ for $a \in \U(n)$; 
\item
$\p^- \Phi_F = 0$, 
where 
\[
\p^- = \left\{
\begin{pmatrix}
A & -\I A \\ -\I A & -A
\end{pmatrix}
\ \middle|\ 
A \in \Sym_2(\C)
\right\}
\subset \g_\infty;
\]
\item
$\Phi_F$ is a cusp form, i.e., 
\[
\int_{N(F) \bs N(\A)} \Phi_F(ng) dg = 0
\]
for $g \in \Sp_n(\A)$ and 
for $N$ being the unipotent radical of any proper $F$-parabolic subgroup $P$ of $\Sp_n$.
\end{enumerate}
Note that $\overline{\rho_{\bk}}$ is isomorphic to the contragredient representation of $\rho_\bk$, 
i.e., 
there exists a non-degenerate bilinear form 
$\pair{\cdot, \cdot}$ on $V_\bk \times V_\bk$ such that
\[
\pair{\rho_\bk(a)v, \rho_{\bk}(\overline{a})v'} = \pair{v,v'}
\]
for $v,v' \in V_\bk$ and $a \in \U(n)$.
Then for $v \in V_\bk$, 
the function
\[
\varphi_{F, v}(g) = \pair{v,\Phi_F(g)} \in \C
\]
is a cusp form on $\Sp_n(\Q) \bs \Sp_n(\A)$.
Note that for fixed $g \in \Sp_n(\A)$, 
the function $K_\infty \ni \kappa_\infty \mapsto \varphi_{F, v}(g \kappa_\infty)$
is a matrix coefficient of $\rho_\bk$.
In particular, the right translations of $\varphi_{F,v}$ under $K_\infty$ for $v \in V_\bk$
form an irreducible representation of $K_\infty$ which is isomorphic to $\rho_\bk$.
We call $\varphi_{F,v}$ a {\it holomorphic cusp form of vector weight $\rho_{\bk}$}.
\par

\subsection{Lowest weight modules}
Let 
\[
\kk_\C = \Lie(K_\infty) \otimes_\R \C
=\left\{
\begin{pmatrix}
A & B \\ -B & A
\end{pmatrix}
\in \Mat_{2n}(\C)
\ \middle|\ 
A=-{}^tA,\ 
B = {}^tB
\right\}
\]
be the complexification of Lie algebra of $K_\infty$.
Fix $\bk = (k_1, \dots, k_n) \in \Z^n$ with $k_1 \geq \dots \geq k_n$.
We denote the differential of the representation $\rho_{\bk}$ of $K_\infty$ by 
$d\rho_{\bk} \colon \kk_\C \rightarrow \mathrm{End}(V_\bk)$.
Since $[\kk_\C, \p^-] \subset \p^-$, 
the representation $d\rho_{\bk}$ can be extended to a representation of $\kk_\C \oplus \p^-$
by setting $d\rho_{\bk}(\p^-) = 0$.
Let $\UU(\g_\infty)$ (\resp $\UU(\kk_\C \oplus \p^-)$) 
be the universal enveloping algebra of $\g_\infty$ (\resp $\kk_\C \oplus \p^-$).
Consider the generalized Verma module 
\[
\MM(V_{\bk}) = \UU(\g_\infty) \otimes_{\UU(\kk_\C \oplus \p^-)} V_{\bk}. 
\]
It is a $(\g_\infty, K_\infty)$-module by setting 
$\kappa_\infty \cdot (1 \otimes v) = 1 \otimes \rho_{\bk}(\kappa_\infty)v$ 
for $\kappa_\infty \in K_\infty$ and $v \in V_\bk$. 
It is known that $\MM(V_{\bk})$ has a unique irreducible quotient, 
which we denote by $L(V_\bk)$.
We call the $(\g_\infty, K_\infty)$-module $L(V_\bk)$ 
the {\it lowest weight module of vector weight $\rho_\bk$}.
This module is characterized so that it contains $V_\bk$ as a $\UU(\kk_\C \oplus \p^-)$-submodule 
with multiplicity one.
The infinitesimal character of $L(V_\bk)$ is given by
\[
(k_1-1,k_2-2, \dots, k_n-n, 0, -(k_n-n), \dots, -(k_2-2), -(k_1-1)).
\]
Note that $L(V_\bk)$ is 
discrete series if $k_n > n$.
\par

Let $\psi \in \Psi_2(\Sp_n/\R)$.
Suppose that 
\[
\psi_d(z) = \psi
\left(z, 
\begin{pmatrix}
(z/\overline{z})^{\half{1}} & 0 \\ 0 & (z/\overline{z})^{-\half{1}}
\end{pmatrix}
\right)
= \begin{pmatrix}
z^{\alpha_1}\overline{z}^{\beta_1}&&\\
&\ddots&\\
&&z^{\alpha_{2n+1}}\overline{z}^{\beta_{2n+1}}\\
\end{pmatrix}
\]
for some $\alpha_i, \beta_i \in \C$ such that $\alpha_i-\beta_i \in \Z$.
Here, for $\alpha, \beta \in \C$ with $\alpha-\beta \in \Z$, 
we write $z^{\alpha}\overline{z}^{\beta} = z^{\alpha-\beta}(z\overline{z})^{\beta}$.
In general, any representation $\pi \in \Pi_\psi$ has an infinitesimal character $(\alpha_1, \dots, \alpha_{2n+1})$.
We call $(\alpha_1, \dots, \alpha_{2n+1})$ the {\it infinitesimal character of $\psi$}.
In particular, if $\Pi_\psi$ contains a discrete series representation, then $\psi$ is Adams--Johnson
(see also \cite[Proposition 4.3, Th\'eor\`eme 4.4]{MR1}).
\par

We determine $A$-packets which contain $L(V_\bk)$.
First, we consider the case where $L(V_\bk)$ is discrete series, i.e., $k_1 \geq \dots \geq k_n > n$.

\begin{prop}\label{discrete}
Suppose that $\bk = (k_1, \dots, k_n) \in \Z^n$ with $k_1 \geq \dots \geq k_n > n$.
Then for $\psi \in \Psi(\Sp_n/\R)$, 
the $A$-packet $\Pi_\psi$ contains the lowest weight module $L(V_\bk)$
if and only if $\psi$ is of the form
\[
\psi = \left(\bigoplus_{i=1}^t \rho_{\alpha_i} \boxtimes S_{d_i}\right) \oplus \sgn^n, 
\]
where
\begin{itemize}
\item
$\sum_{i=1}^{t} d_i = n$; 
\item
$\alpha_i + d_i \equiv 1 \bmod 2$; 
\item $\alpha_1 > \dots > \alpha_t >0$ and 
\[
\bigcup_{i=1}^t 
\left\{ \half{\alpha_i+d_i-1}, \half{\alpha_i+d_i-3}, \dots, \half{\alpha_i-d_i+1} \right\}
= \{k_1-1,k_2-2, \dots, k_n-n\}.
\]
\end{itemize}
In this case, the character $\pair{\cdot, L(V_\bk)}_\psi$ of $A_\psi$ 
is determined by $\pair{z_\psi, L(V_\bk)}_\psi = 1$ and 
\[
\pair{\alpha_{\rho_{\alpha_i} \boxtimes S_{d_i}}, L(V_\bk)}_\psi = (-1)^{\half{d_i-\delta_i}}
\]
for $i = 1, \dots, t$, where $\delta_i$ is given in Theorem \ref{AJ}.
\end{prop}
\begin{proof}
By the above remark, if $L(V_\bk) \in \Pi_\psi$, then $\psi$ is Adams--Johnson.
It is known that 
the cohomological induction $\pi_w = A_{w^{-1}Qw}(w^{-1}\chi_{\lam_\psi})$ is discrete series 
if and only if $w^{-1}Lw(\R)$ is compact.
It happens for some $w \in \Sigma_\psi$ 
only when $\sgn^n \subset \psi$, i.e., $d_0 = 1$ in the notation in Theorem \ref{AJ}.
By comparing the infinitesimal character, 
we see that if $L(V_\bk) \in \Pi_\psi$, then $\psi$ is of the form in the proposition.
\par

Suppose that $\psi$ is of the form in the proposition.
Then we have $\pi_1 = A_{Q}(\chi_{\lam_\psi}) \cong L(V_\bk)$.
The associated character $\pair{\cdot, L(V_\bk)}_\psi$ of $A_\psi$ is computed in Theorem \ref{AJ}.
\end{proof}
\par

When $\bk = (k,k,\dots,k)$, 
i.e., when $(\rho_{\bk}, V_\bk) = ({\det}^k, \C)$, 
we write $\DD_k^{(n)} = L(V_\bk)$
and call it the {\it lowest weight module of the scalar weight $k$}.
Some properties of $\psi \in \Psi(\Sp_n/\R)$ satisfying $\DD^{(n)}_k \in \Pi_\psi$
are established by M{\oe}glin--Renard \cite{MR} even when $\DD_k^{(n)}$ is not discrete series.

\begin{prop}\label{MR}
If $0 < k \leq n$, then the multiplicity of $\DD^{(n)}_k$ in $\Pi_\psi$ is at most one.
Moreover, for $k_1, k_2 > 0$ with $k_1 \not= k_2$, 
if $\DD^{(n)}_{k_1}$ and $\DD^{(n)}_{k_2}$ belong to the same $A$-packet $\Pi_\psi$, 
then $n$ is even and $\{k_1,k_2\} = \{\half{n}, \half{n}+1\}$.
\end{prop}
\begin{proof}
The first assertion is a part of \cite[Th{\'e}or{\`e}me 6.1]{MR}.
If $\DD^{(n)}_k \in \Pi_\psi$, 
then the infinitesimal character of $\DD^{(n)}_k$ is determined by the $A$-parameter $\psi$.
If $k > n$, since the infinitesimal character of $\DD^{(n)}_k$ is regular, 
$\psi$ determines $\DD^{(n)}_k$ uniquely.
Hence we may assume that $0 < k_1, k_2 \leq n$.
In this case, the last assertion follows from \cite[Th{\'e}or{\`e}me 6.2]{MR}.
\end{proof}

\subsection{Automorphic representations generated by Siegel modular forms}
In this subsection, we let $\F = \Q$.
We prove some lemmas.

\begin{lem}\label{irred}
Let $F \in S_{\rho_\bk}(\Sp_n(\Z))$ be a Hecke eigenform with 
Satake parameter $c_p(F) \in (\C^\times)^n/\mathfrak{S_n} \ltimes \{\pm1\}^n$ at $p$.
Then the cuspidal automorphic representation $\pi_F \subset \AA^2(\Sp_n(\A))$ 
generated by the cusp form $\varphi_{F,v}$ for $v \in V_{\bk}$ is irreducible.
Moreover, its local factors $\pi_{F,p}$ and $\pi_{F,\infty}$ are given as follows:
\begin{itemize}
\item
$\pi_{F,p}$ is unramified with Satake parameter $c(\pi_{F,p}) = c_p(F)$ for each (finite) prime $p$; 
\item
$\pi_{F, \infty}$ is the lowest weight module $L(V_\bk)$.
\end{itemize}
\end{lem}
\begin{proof}
The proof is the same as \cite[Theorem 4.3]{AK}.
\par

Since $\pi_{F}$ is cuspidal, it is a direct sum of irreducible automorphic representations.
Fix an isomorphism $\pi_F \cong \pi_1 \oplus \dots \oplus \pi_r$ with $\pi_i \cong \otimes_v' \pi_{i,v}$ irreducible.
Let $\varphi_{i,v}$ be the image of $\varphi_{F,v} \in \pi_F$ for $v \in V_{\bk}$ 
under the projection $\pi_F \cong \oplus_{j=1}^r \pi_j \twoheadrightarrow \pi_j$.
Since $\pi_F$ is generated by $\{\varphi_{F,v}|v \in V_{\bk}\}$, 
for each $i = 1, \dots, r$, some $\varphi_{i,v}$ is nonzero.
By considering the action of the spherical Hecke algebra $\mathcal{H}(\Sp_n(\Q_p), \Sp_n(\Z_p))$ 
on $\varphi_{i,v}$, 
we see that $\pi_{i,p}$ is unramified with Satake parameter $c(\pi_{i,p}) = c_p(F)$ for each prime $p$.
Moreover, since $\kappa \cdot \varphi_{F,v} = \varphi_{F, \kappa v}$ for $\kappa \in K_\infty$, 
for each $i = 1, \dots, r$, we have $\kappa \cdot \varphi_{i,v} = \varphi_{i, \kappa v}$ for $\kappa \in K_\infty$.
This means that there is a nonzero $K_\infty$-intertwining operator 
$V_\bk \hookrightarrow \pi_{i,\infty} \cong \pi_{i}^{K_\fin}$.
Hence $\pi_{i,\infty} \cong L(V_{\bk})$ since $\pi_{i,\infty}$ is irredcuible.
\par

We conclude that $\pi_F$ is isotypic, i.e., $\pi_F \cong \pi \otimes \C^r$ 
for some irreducible automorphic representation $\pi$.
However, since $\pi_F^{K_\fin} \cong \pi^{K_\fin} \otimes \C^r$ 
is generated by $\{\varphi_{F,v}|v \in V_{\bk}\}$ as $(\g_\infty, K_\infty)$-module, 
it must be irreducible.
Hence we have $r=1$ so that $\pi_F$ is irreducible.
\end{proof}
\par

Now we consider the elliptic cusp forms.
Let $f \in S_{2k}(\SL_2(\Z))$ be a Hecke eigenform.
Consider the irreducible unitary cuspidal automorphic representation $\tau_f$ of $\GL_2(\A)$
generated by $\varphi_f$, 
where $\varphi_f$ is a cusp form on $\GL_2(\A)$ defined by
\[
\varphi_f(\gamma g_\infty \kappa_\fin) = \det(g_\infty)^{k} f(g_\infty \pair{\I}) j(g_\infty, \I)^{-2k}
\]
for $\gamma \in \GL_2(\Q)$, $g_\infty \in \GL_2(\R)$ with $\det(g_\infty) > 0$, 
and $\kappa_\fin \in \GL_2(\widehat{\Z})$.
Note that $\tau_f$ is symplectic, i.e., $L(s,\tau_f,\wedge^2)$ has a pole at $s=1$
since the central character of $\tau_f$ is trivial (see \cite[Corollary 7.5]{KR}).
Moreover, the representation of $W_\R$ corresponding to $\tau_{f,\infty}$ is $\rho_{2k-1}$.
\par

In the rest of this subsection, 
we consider the Siegel modular forms of degree $2$.
When $\bk = (k+j, k)$ with $j \geq 0$, 
we have $\rho_{\bk} = {\det}^k \Sym(j)$, whose dimension is $j+1$.
In this case, we write $S_{k,j}(\Sp_2(\Z)) = S_{\rho_{(k+j,k)}}(\Sp_2(\Z))$, 
and call $f \in S_{k,j}(\Sp_2(\Z))$ 
a {\it Siegel modular form of degree $2$ and of weight ${\det}^k \Sym(j)$}.
Note that $S_{k,j}(\Sp_2(\Z)) = 0$ unless $j \equiv 0 \bmod 2$.
We assume that $j$ is even in the rest of this subsection.
For a Hecke eigenform $f \in S_{k,j}(\Sp_2(\Z))$, 
we can define a degree four $L$-function 
\[
L(s,f,\spin) = \prod_p L_p(s,f,\spin)
\]
which is called the {\it spinor $L$-function associated to $f$}.
\par

We would prove the following lemma.
\begin{lem}\label{nonSK}
Assume $k \geq 4$ and $j \geq 1$. 
Let $f \in S_{k,j}(\Sp_2(\Z))$ be a Hecke eigenform.
Then there is an irreducible unitary cuspidal symplectic automorphic representation 
$\tau_f = \otimes_v' \tau_{f,v}$ of $\GL_4(\A)$ 
such that 
\[
L(s, \tau_{f,p}) = L_p(s, f, \spin)
\]
for any prime $p < \infty$, 
and the $L$-parameter of $\tau_{f,\infty}$ is $\rho_{j+2k-3} \oplus \rho_{j+1}$.
\end{lem}

To prove this lemma, 
we use Arthur's multiplicity formula for $\SO(3,2)$
via an accidental isomorphism $\mathrm{PGSp}_2 \cong \SO(3,2)$, 
where $\SO(3,2)$ is the split special orthogonal group of type $B_2$ and of discriminant $1$.
We omit the detail for Arthur's multiplicity formula for $\SO(3,2)$.
See Arthur's book \cite{Ar}.
\par

For $v \in V_{(k+j,k)}$, one can define a cusp form $\varphi_{f,v}$
on $\mathrm{PGSp}_2(\Q) \bs \mathrm{PGSp}_2(\A)$ by a similar way to the elliptic modular forms case.
We regard $\varphi_{f,v}$ as a cusp form on $\SO(3,2)(\Q) \bs \SO(3,2)(\A)$
via an accidental isomorphism $\mathrm{PGSp}_2 \cong \SO(3,2)$.
By a similar argument to the proof of Lemma \ref{irred}, 
we see that $\varphi_{f,v}$ for $v \in V_{(k+j,k)}$ 
generates an irreducible cuspidal automorphic representation $\sigma_f$ of $\SO(3,2)(\A)$.
One can associate a global $A$-parameter $\psi_f \in \Psi_2(\SO(3,2)/\Q)$ to $\sigma_f$.
Note that $\sigma_{f,\infty}$ is a discrete series representation, 
and its $L$-parameter is $\rho_{j+2k-3} \oplus \rho_{j+1}$.
Comparing the infinitesimal character, we see that 
\[
(\psi_{f,\infty})_d \cong \chi_{j+2k-3} \oplus \chi_{j+1} \oplus \chi_{-(j+1)} \oplus \chi_{-(j+2k-3)}. 
\]
\par

An $A$-parameter $\psi \in \Psi_2(\SO(3,2)/\Q)$ is one of the following forms:
\begin{enumerate}
\item
$\psi = \chi[4]$, where $\chi$ is a quadratic character of $\A^\times/\Q^\times$; 
\item
$\psi = \chi_1[2] \boxplus \chi_2[2]$, 
where $\chi_1$, $\chi_2$ are quadratic characters of $\A^\times/\Q^\times$; 
\item
$\psi = \tau[2]$, where $\tau$ is an irreducible cuspidal orthogonal representation of $\GL_2(\A)$; 
\item
$\psi = \chi[2] \boxplus \tau[1]$, 
where $\chi$ is a quadratic character of $\A^\times/\Q^\times$
and $\tau$ is an irreducible cuspidal symplectic representation of $\GL_2(\A)$; 
\item
$\psi = \tau_1[1] \boxplus \tau_2[1]$, 
where $\tau_1$, $\tau_2$ are irreducible cuspidal symplectic representations of $\GL_2(\A)$; 
\item
$\psi = \tau[1]$, 
where $\tau$ is an irreducible unitary cuspidal symplectic automorphic representation of $\GL_4(\A)$.
\end{enumerate}
\par

When $\psi = \psi_f$, 
the representations $\chi$ or $\tau$ appearing in $\psi$ should be unramified everywhere.
In particular, $\chi$ must be the trivial character of $\A^\times/\Q^\times$
since $\A^\times/\Q^\times \widehat{\Z}^\times \R^\times_{>0} = \{1\}$.
When $\psi = \1[4]$, 
the global $A$-packet $\Pi_\psi$ consists of the trivial representation of $\SO(3,2)(\A)$.
Since $f \in S_{k,j}(\Sp_2(\Z))$ is not a constant function, $\psi_f \not\cong \1[4]$.
Hence the case (1) cannot occur.
The case (2) is impossible since $\1[2] \boxplus \1[2]$ is not a discrete $A$-parameter. 
\par

Now consider $\psi = \tau[2]$ is in the case (3).
Then 
$(\psi_{\infty})_d \cong \chi_{\alpha+1} \oplus \chi_{\alpha-1} \oplus \chi_{-\alpha+1} \oplus \chi_{-\alpha-1}$
for some $\alpha \in \Z$ with $\alpha \geq 0$.
If $\psi_f = \psi$, we have $\alpha+1 = j+2k-3$ and $\alpha-1 = j+1$.
This implies that $k=3$, which contradicts to our assumption that $k \geq 4$.
\par

Next, we consider $\psi = \1[2] \boxplus \tau[1]$ as in the case $(4)$.
Then 
$(\psi_{\infty})_d \cong \chi_{\alpha} \oplus \chi_{1} \oplus \chi_{-1} \oplus \chi_{-\alpha}$
for some $\alpha \in \Z$ with $\alpha \geq 0$.
If $\psi_f = \psi$, we have $j=0$ an $\alpha = 2k-3$.
This contradicts to our assumption that $j \geq 1$.
Remark that this is compatible that there is no Saito--Kurokawa lifting of vector weight.
\par

Suppose that $\psi = \tau_1[1] \boxplus \tau_2[1]$ as in the case $(5)$
such that $\psi_{\infty} = \rho_{j+2k-3} \oplus \rho_{j+1}$.
Consider the representation $\sigma = \otimes_v' \sigma_v \in \Pi_\psi$ 
such that $\sigma_p$ is unramified at any $p < \infty$, 
and $\sigma_\infty \cong \sigma_{f,\infty}$.
Since $\sigma_{f,\infty}$ is not generic, 
the character $\pair{\cdot, \sigma_{\infty}}_{\psi_{\infty}}$ is not trivial.
However, the Arthur's character $\ep_{\psi}$ must be trivial.
This implies that $\sigma$ is not automorphic by Arthur's multiplicity formula for $\SO(3,2)$.
Remark that this is compatible that there is no Yoshida lifting of level one.
\par

In conclusion, 
if $k \geq 4$ and $j \geq 1$, 
the $A$-parameter $\psi_f \in \Psi_2(\SO(3,2)/\Q)$ is $\psi_f = \tau_f[4]$ as in the case (6).
By the compatibility of Satake parameters, we see that
\[
L(s, \tau_{f,p}) = L_p(s, f, \spin)
\]
for any $p < \infty$.
This completes the proof of Lemma \ref{nonSK}.

\subsection{Proof of Lifting Theorem}\label{sec.LT}
Now we prove Lifting Theorem (Theorem \ref{LT}).
Let $\F = \Q$.
\par

For fixed $\bk = (k_1, \dots, k_n) \in \Z^n$ with $k_1 \geq \dots \geq k_n > n$, 
we consider a Hecke eigenform $g \in S_{\rho_\bk}(\Sp_n(\Z))$.
Let $\psi_g = \boxplus_{i=1}^r \tau_i[d_i] \in \Psi_2(\Sp_n/\Q)$ 
be the $A$-parameter associated to $g$, 
i.e., $\pi_g \in \Pi_{\psi_g}$, 
where $\pi_{g}$ is the irreducible cuspidal automorphic representation
generated by the cusp form $\varphi_{g,v}$ for $v \in V_{\bk}$.
Since $\pi_{g, \fin}$ is unramified everywhere, 
and $\pi_{g,\infty} \cong L(V_\bk)$, 
by Arthur's multiplicity formula (Theorem \ref{AMF}), we have
\[
\ep_{\psi_g}(\alpha_{\tau_i[d_i]})
= \pair{\Delta(\alpha_{\tau_i[d_i]}), \pi_g}_{\psi_{g}} 
= \pair{\alpha_{\phi_{i, \infty} \boxtimes S_{d_i}}, L(V_\bk)}_{\psi_{g,\infty}}
\]
for $i=1, \dots, r$, where $\phi_{i,\infty}$ is the representation of $W_\R$
corresponding to $\tau_{i,\infty}$.
Moreover, since $L(V_\bk)$ is discrete series, by Proposition \ref{discrete}, 
the localization $\psi_{g,\infty}$ at $\infty$ is Adams--Johnson and is of the form
\[
\psi_{g,\infty} = \left( \bigoplus_{l=1}^{t} \rho_{\alpha_l} \boxtimes S_{d_j} \right) \oplus \sgn^n, 
\]
where $\alpha_1 > \dots > \alpha_t > 0$ and 
\begin{align*}
&\bigcup_{l=1}^t \left\{ \half{\alpha_l+d_l-1}, \half{\alpha_l+d_l-3}, \dots, \half{\alpha_l-d_l+1} \right\}
= \left\{ k_1-1, k_2-2, \dots, k_n-n \right\}.
\end{align*}
In particular, there is a unique index $i_0$ such that 
\[
\phi_{i_0,\infty} \cong \left(\bigoplus_{l=1}^{m_{i_0}} \rho_{\alpha_{i_0,l}}\right) \oplus \sgn^n 
\]
with even $\alpha_{i_0,l}$ (and $d_{i_0} = 1$), 
and for $i \not= i_0$, 
\begin{align*}
\phi_{i,\infty} \cong  \bigoplus_{l=1}^{m_i} \rho_{\alpha_{i,l}}
\end{align*}
with $\alpha_{i,l} \not\equiv d_i \bmod 2$.
Since $\det(\phi_{i,\infty})(-1) = \omega_{\tau_i}(-1) = 1$, 
where $\omega_{\tau_i}$ is the central character of $\tau_i$, 
we have $(-1)^{m_{i_0}+n} = 1$ and $(-1)^{m_id_i} = 1$ for $i \not= i_0$.
Moreover, by Proposition \ref{discrete}, for $i \not= i_0$, we have
$\pair{\alpha_{\phi_{i, \infty} \boxtimes S_{d_i}}, L(V_\bk)}_{\psi_{g,\infty}} = (-1)^{\half{m_id_i}}$.
\par

Now we prove Lifting Theorem (A).
\begin{proof}[Proof of Theorem \ref{LT} (A)]
Let $f \in S_{2k}(\SL_2(\Z))$ be a Hecke eigenform, 
and $\tau_f$ be the irreducible cuspidal symplectic automorphic representation of $\GL_2(\A)$.
Take an integer $d > 0$ with $k > d$ such that $k+d-1 < k_n-n$ or $k-d > k_1-1$.
This condition implies that $\tau_{f, \infty} \not\cong \tau_{i,\infty}$ for any $1 \leq i \leq r$.
We set
\[
\psi_{f,g} = \tau_f[2d] \boxplus \psi_g.
\]
One can easily see that $\psi_{f,g} \in \Psi_2(\Sp_{n+2d}/\Q)$.
Moreover, since $\psi_{f,g,p}$ is unramified for any (finite) prime $p$, 
the local $A$-packet $\Pi_{\psi_{f,g,p}}$ contains a unique unramified representation $\pi_{f,g,p}$.
In addition, since
\[
\psi_{f,g,\infty} \cong (\rho_{2k-1} \boxtimes S_{2d}) \oplus
\left( \bigoplus_{l=1}^{t} \rho_{\alpha_l} \boxtimes S_{d_j} \right) \oplus \sgn^{n+2d} 
\]
is Adams--Johnson, 
the local $A$-packet $\Pi_{\psi_{f,g,\infty}}$ contains $L(V_{\bk'})$, 
where $\bk' = (k_1', \dots, k_{n+2d}') \in \Z^{n+2d}$ is given so that $k_1' \geq \dots \geq k_{n+2d}'$ and 
\begin{align*}
&\left\{ k_1'-1, k_2'-2, \dots, k_{n+2d}'-n-2d \right\}
\\&=
\left\{ k_1-1, k_2-2, \dots, k_n-n \right\}
\cup
\left\{
k+d-1, k+d-2, \dots, k-d
\right\}. 
\end{align*}
\par

Let us consider 
\[
\pi_{f,g} = 
\left( {\bigotimes_{p<\infty}}' \pi_{f,g,p} \right) \otimes L(V_{\bk'}) \in \Pi_{\psi_{f,g}}.
\]
Using Arthur's multiplicity formula (Theorem \ref{AMF}), 
we study when $\pi_{f,g}$ is automorphic.
It is automorphic if and only if 
\[
\ep_{\psi_{f,g}}(\alpha_{\tau_f[2d]})
= \pair{\Delta(\alpha_{\tau_f[2d]}), \pi_{f,g}}_{\psi_{f,g}}. 
\]
and 
\[
\ep_{\psi_{f,g}}(\alpha_{\tau_i[d_i]})
= \pair{\Delta(\alpha_{\tau_i[d_i]}), \pi_{f,g}}_{\psi_{f,g}} 
\]
for $i=1, \dots, r$.
\par

Before checking these conditions, we compute the root number $\ep(\tau_i \times \tau_f)$.
Since $\tau_{i,\fin}$ and $\tau_{f,\fin}$ are unramified everywhere, 
$\ep(\tau_i \times \tau_f)$ is equal to the local root number $\ep(\phi_{i,\infty} \otimes \rho_{2k-1})$. 
For $i = i_0$, we have
\begin{align*}
\ep(\phi_{i_0,\infty} \otimes \rho_{2k-1})
&= \left( \prod_{l=1}^{m_{i_0}} \ep(\rho_{\alpha_{i_0,l}} \otimes \rho_{2k-1}) \right) 
\ep(\sgn^n \otimes \rho_{2k-1})
\\&= 
\left\{
\begin{aligned}
&(-1)^{m_{i_0}+k} \iif k+d-1 < k_n-n, \\
&(-1)^k \iif k-d > k_1-1.
\end{aligned}
\right.
\end{align*}
For $i \not= i_0$, we have
\begin{align*}
\ep(\phi_{i,\infty} \otimes \rho_{2k-1})
&= \left( \prod_{l=1}^{m_i} \ep(\rho_{\alpha_{i,l}} \otimes \rho_{2k-1}) \right)
= 
\left\{
\begin{aligned}
&(-1)^{m_{i}d_i} \iif k+d-1 < k_n-n, \\
&1 \iif k-d > k_1-1,
\end{aligned}
\right.
\end{align*}
which is equal to $1$.
\par

Hence we have
\begin{align*}
\ep_{\psi_{f,g}}(\alpha_{\tau_f[2d]}) 
&= \prod_{i=1}^r \ep(\tau_i \times \tau_f)^{\min\{d_i, 2d\}}
\\&= (-1)^k \times 
\left\{
\begin{aligned}
&(-1)^n \iif k+d-1 < k_n-n, \\
&1 \iif k-d > k_1-1.
\end{aligned}
\right.
\end{align*}
On the other hand
\begin{align*}
\pair{\Delta(\alpha_{\tau_f[2d]}), \pi_{f,g}}_{\psi_{f,g}}
&= 
\pair{\alpha_{\rho_{2k-1} \otimes S_{2d}}, L(V_{\bk'})}
= (-1)^{d}.
\end{align*}
Therefore, the condition 
$\ep_{\psi_{f,g}}(\alpha_{\tau_f[2d]}) = \pair{\Delta(\alpha_{\tau_f[2d]}), \pi_{f,g}}_{\psi_{f,g}}$
is equivalent that 
\[
k \equiv 
\left\{
\begin{aligned}
&d+n &&\bmod 2\iif k+d-1 < k_n-n, \\
&d &&\bmod 2 \iif k-d > k_1-1.
\end{aligned}
\right.
\]
\par

Next, we check the condition 
$\ep_{\psi_{f,g}}(\alpha_{\tau_i[d_i]}) = \pair{\Delta(\alpha_{\tau_i[d_i]}), \pi_{f,g}}_{\psi_{f,g}}$
for $i \not= i_0$.
We have 
\begin{align*}
\pair{\Delta(\alpha_{\tau_i[d_i]}), \pi_{f,g}}_{\psi_{f,g}} 
&= \pair{\alpha_{\phi_{i,\infty} \boxtimes S_{d_i}}, L(V_{\bk'})}_{\psi_{f,g,\infty}} 
\\&= \pair{\alpha_{\phi_{i,\infty} \boxtimes S_{d_i}}, L(V_{\bk})}_{\psi_{g,\infty}}
\\&= \ep_{\psi_{g}}(\alpha_{\tau_i[d_i]})
\\&= \ep_{\psi_{f,g}}(\alpha_{\tau_i[d_i]}) \cdot \ep(\tau_i \times \tau_f)^{-\min\{d_i,2d\}}
= \ep_{\psi_{f,g}}(\alpha_{\tau_i[d_i]}), 
\end{align*}
as desired.
\par

Finally, since 
\[
\left( \prod_{i=1}^r \ep_{\psi_{f,g}}(\alpha_{\tau_i[d_i]}) \right) \cdot \ep_{\psi_{f,g}}(\alpha_{\tau_f[2d]}) = 1
\]
and 
\[
\left(\prod_{i=1}^r \pair{\Delta(\alpha_{\tau_i[d_i]}), \pi_{f,g}}_{\psi_{f,g}} \right) 
\cdot \pair{\Delta(\alpha_{\tau_f[2d]}), \pi_{f,g}}_{\psi_{f,g}} = 1, 
\]
when $k \equiv d+n \bmod 2$ if $k+d-1 < k_n-n$ (\resp when $k \equiv d \bmod 2$ if $k-d > k_1-1$), 
we have
$\ep_{\psi_{f,g}}(\alpha_{\tau_i[d_i]}) = \pair{\Delta(\alpha_{\tau_i[d_i]}), \pi_{f,g}}_{\psi_{f,g}}$
for $i = i_0$.
\par

In conclusion, $\pi_{f,g}$ is automorphic if and only if 
$k \equiv d+n \bmod 2$ when $k+d-1 < k_n-n$ (\resp $k \equiv d \bmod 2$ when $k-d > k_1-1$).
In this case, 
since $k'_{n+2d}-(n+2d) = \min\{k-d, k_n-n\} > 0$ 
so that $\pi_{f,g,\infty} \cong L(V_{\bk'})$ is discrete series, 
by a result of Wallach \cite{W}, we see that $\pi_{f,g}$ is cuspidal.
Let $\pi_{f,g}^{K_\fin, \p^-}$ be the subspace of $\pi_{f,g}$ 
on which $K_\fin = \prod_{p < \infty} \Sp_{n+2d}(\Z_p)$ acts by trivially and $\p^-$ acts by zero.
Note that $\pi_{f,g}^{K_\fin, \p^-}$ is isomorphic to $\rho_{\bk'}$ as a representation of $K_\infty$.
Fix a nonzero function $\varphi \in \pi_{f,g}^{K_\fin, \p^-}$.
For $v \in V_{\bk'}$, consider the integral
\[
\Phi_v(g) = \int_{K_\infty} \varphi(g \kappa) \overline{\rho_{\bk'}}(\kappa)v d \kappa.
\]
By the Schur orthogonality relations, 
$\Phi_v$ is nonzero for some $v \in V_{\bk'}$.
Since $\Phi_v$ satisfies that $\Phi_v(g \kappa_\infty) = \overline{\rho_{\bk'}}(\kappa_\infty)^{-1} \Phi_v(g)$, 
it gives a Siegel modular form $F_{f,g} \in S_{\rho_{\bk'}}(\Sp_{n+2d}(\Z))$.
This is a Hecke eigenform such that 
\[
L(s,F_{f,g}, \std) = L(s,g,\std) \prod_{i=1}^{2d}L(s+k+d-i,f).
\]
This completes the proof of Lifting Theorem (A) (Theorem \ref{LT} (A)).
\end{proof}

\begin{rem}\label{Ib2}
Recall that the $A$-parameter $\psi_g$ associated to $g$ satisfies that
\[
\ep_{\psi_g}(\alpha_{\tau_i[d_i]}) 
= \pair{\alpha_{\phi_{i, \infty} \boxtimes S_{d_i}}, L(V_\bk)}_{\psi_{g,\infty}} = (-1)^{\half{m_id_i}}
\]
for $i \not= i_0$.
In particular, if $\psi_g$ is tempered, 
then we must have $m_id_i \equiv 0 \bmod 4$ for $i \not= i_0$.
When $n = 2$, i.e., $g \in S_{\rho_\bk}(\Sp_2(\Z))$, 
the $A$-parameter $\psi_g$ does not of the form $\tau[1] \oplus \chi[1]$
with $\tau$ being an irreducible cuspidal orthogonal automorphic representation of $\GL_4(\A)$.
Using this fact, one can prove Ibukiyama's conjecture of Type $\mathrm{II}$ in Introduction (\S \ref{intro})
by a similar argument to the proof of Lifting Theorem (A).
\end{rem}
\par

Next, we prove Lifting Theorem (B).
\begin{proof}[Proof of Theorem \ref{LT} (B)]
Assume that $k \geq 4$ and $j \geq 1$.
Let $f \in S_{k,j}(\Sp_2(\Z))$ be a Hecke eigenform, 
and $\tau_f$ be the associated irreducible cuspidal symplectic automorphic representation of $\GL_4(\A)$
by Lemma \ref{nonSK}.
Take an integer $d > 0$ such that $d < \min\{(k/2)-1, (j/2)+1\}$ and that 
$k_i-i \not\in [(j/2)-d+1, (j/2)+k+d-2]$
for each $i = 1, \dots, n$.
The last condition implies that $\tau_{f, \infty} \not\cong \tau_{i,\infty}$ for any $1 \leq i \leq r$.
We set
\[
\psi_{f,g} = \tau_f[2d] \boxplus \psi_g.
\]
One can easily see that $\psi_{f,g} \in \Psi_2(\Sp_{n+4d}/\Q)$.
Moreover, since $\psi_{f,g,p}$ is unramified for any (finite) prime $p$, 
the local $A$-packet $\Pi_{\psi_{f,g,p}}$ contains a unique unramified representation $\pi_{f,g,p}$.
In addition, since 
\[
\psi_{f,g,\infty} \cong (\rho_{j+2k-3} \boxtimes S_{2d}) \oplus (\rho_{j+1} \boxtimes S_{2d}) \oplus
\left( \bigoplus_{j=1}^{t} \rho_{\alpha_j} \boxtimes S_{d_j} \right) \oplus \sgn^{n+4d} 
\]
is Adams--Johnson, 
the local $A$-packet $\Pi_{\psi_{f,g,\infty}}$ contains $L(V_{\bk'})$, 
where $\bk' = (k_1', \dots, k_{n+4d}') \in \Z^{n+4d}$ is given by $k_1' \geq \dots \geq k_{n+4d}'$ and 
\begin{align*}
&\left\{ k_1'-1, k_2'-2, \dots, k_{n+4d}'-n-4d \right\}
\\&=
\left\{ k_1-1, k_2-2, \dots, k_n-n \right\}
\\&\cup
\left\{
\half{j}+k+d-2, \half{j}+k+d-3, \dots, \half{j}+k-d-1
\right\}
\cup
\left\{
\half{j}+d, \half{j}+d-1, \dots, \half{j}-d+1
\right\}.
\end{align*}
\par

Let us consider 
\[
\pi_{f,g} = 
\left( {\bigotimes_{p<\infty}}' \pi_{f,g,p} \right) \otimes L(V_{\bk'}) \in \Pi_{\psi_{f,g}}.
\]
Using Arthur's multiplicity formula (Theorem \ref{AMF}), 
we study when $\pi_{f,g}$ is automorphic.
It is automorphic if and only if 
\[
\ep_{\psi_{f,g}}(\alpha_{\tau_f[2d]})
= \pair{\Delta(\alpha_{\tau_f[2d]}), \pi_{f,g}}_{\psi_{f,g}}. 
\]
and 
\[
\ep_{\psi_{f,g}}(\alpha_{\tau_i[d_i]})
= \pair{\Delta(\alpha_{\tau_i[d_i]}), \pi_{f,g}}_{\psi_{f,g}} 
\]
for $i=1, \dots, r$.
\par

Note that $\alpha_{i,l} \not\in [j+1, j+2k-3]$ for any $(i,l)$ 
since $k_i-i \not\in [(j/2)-d+1, (j/2)+k+d-2]$
for any $i = 1, \dots, n$.
By a similar argument to the proof of Lifting Theorem (A), 
we have 
\begin{align*}
\ep(\tau_i \times \tau_f) = 
\left\{
\begin{aligned}
&(-1)^{j+k} \iif i = i_0, \\
&1 \iif i \not= i_0.
\end{aligned}
\right.
\end{align*}
Hence we have
$\ep_{\psi_{f,g}}(\alpha_{\tau_f[2d]}) = (-1)^{j+k} = (-1)^k$ since $j$ is even.
On the other hand, 
\[
\pair{\Delta(\alpha_{\tau_f[2d]}), \pi_{f,g}}_{\psi_{f,g}}
=
\pair{\alpha_{\rho_{j+2k-3} \otimes S_{2d}}, L(V_{\bk'})}_{\psi_{f,g,\infty}}
\pair{\alpha_{\rho_{j+1} \otimes S_{2d}}, L(V_{\bk'})}_{\psi_{f,g,\infty}}
=1.
\]
Hence the condition 
$\ep_{\psi_{f,g}}(\alpha_{\tau_f[2d]}) = \pair{\Delta(\alpha_{\tau_f[2d]}), \pi_{f,g}}_{\psi_{f,g}}$
is equivalent that $k$ is even.
\par

For $i \not= i_0$, since $\ep(\tau_i \times \tau_f) = 1$, 
we have 
$\ep_{\psi_{f,g}}(\alpha_{\tau_i[d_i]}) = \pair{\Delta(\alpha_{\tau_i[d_i]}), \pi_{f,g}}_{\psi_{f,g}}$.
Therefore, this condition also holds for $i=i_0$ when $k$ is even.
\par

In conclusion, $\pi_{f,g}$ is automorphic if and only if $k \equiv j \equiv 0 \bmod 2$.
In this case, 
since $\pi_{f,g,\infty} \cong L(V_{\bk'})$ is discrete series, 
by a result of Wallach \cite{W}, we see that $\pi_{f,g}$ is cuspidal.
The space $\pi_{f,g}^{K_\fin, \p^-}$
gives a Hecke eigenform $F_{f,g} \in S_{\rho_{\bk'}}(\Sp_{n+4d}(\Z))$ 
such that 
\[
L(s,F_{f,g}, \std) = L(s,g,\std) \prod_{i=1}^{2d}L \left(s+d+\half{1}-i,f, \spin \right).
\]
This completes the proof of Lifting Theorem (B) (Theorem \ref{LT} (B)).
\end{proof}

\subsection{Arthur's multiplicity formula for holomorphic cusp forms}\label{sec.SMO}
Let $\F$ be a totally real field.
For each infinite place $v \mid \infty$, 
we consider the standard maximal compact group $K_v \subset \Sp_n(\F_v) = \Sp_n(\R)$
and the subalgebra $\p^-_v \subset \Lie(\Sp_n(\F_v)) \otimes_\R \C$ defined in \S \ref{sec.Siegel}.

\begin{defi}
Let $k = (k_v)_v \in \prod_{v \mid \infty}\Z_{>0}$.
A cusp form $\varphi \colon \Sp_n(\A) \rightarrow \C$ is called {\it holomorphic cusp form of weight $k$} 
if for each $v \mid \infty$, 
\begin{enumerate}
\item
$\varphi(g\kappa_v) = \det(\kappa_v)^{k_v} \varphi(g)$ for $\kappa_v \in K_v$ and $g \in \Sp_n(\A)$; 
\item
$\p_v^- \varphi = 0$.
\end{enumerate}
Here, we identify $K_v$ with $\U(n)$, and 
we denote by $\det$ the determinant of $K_v \cong \U(n)$.
\end{defi}
The space of holomorphic cusp forms of weight $k$ is denoted by $\Sc_k(\Sp_n(\A))$.
This is an $\Sp_n(\A_\fin)$-subrepresentation of $\AA^2(\Sp_n(\A))$.
Note that $\g_\infty$ does not act on $\Sc_k(\Sp_n(\A))$.
\par

\begin{ex}
Suppose that $\F = \Q$.
Then the Siegel modular form $F \in S_k(\Sp_n(\Z))$
gives a cusp form $\varphi_F \in \AA^2(\Sp_n(\A))$.
As explained in \S \ref{sec.Siegel}, 
$\varphi_F$ is a holomorphic cusp form of weight $k$.
\end{ex}
\par

For $k = (k_v) \in \prod_{v \mid \infty}\Z_{>0}$, 
we set $\Psi_2(\Sp_n/\F, \DD^{(n)}_k)$ to be the subset of $\Psi_2(\Sp_n/\F)$ 
consisting of global $A$-parameters $\psi$ such that
$\DD^{(n)}_{k_v} = L(V_{(k_v, \dots, k_v)}) \in \Pi_{\psi_v}$ for any $v \mid \infty$.
For $\psi \in \Psi_2(\Sp_n/\F,\DD^{(n)}_k)$, define 
\[
\Pi_\psi^{\fin} = \left\{\pi_\fin = {\bigotimes_{v < \infty}}' \pi_v
\ \middle|\ 
\pi_v \in \Pi_{\psi_v}, \ 
\text{$\pair{\cdot, \pi_v}_{\psi_v} = \1$ for almost all $v$}
\right\}.
\]
\par

Now we obtain Arthur's multiplicity formula for holomorphic cusp forms.
\begin{thm}\label{AMF-S}
For $k = (k_v) \in \prod_{v \mid \infty}\Z_{>0}$, 
we have
\[
\Sc_{k}(\Sp_n(\A)) \subseteq 
\bigoplus_{\psi \in \Psi_2(\Sp_n/\F, \DD^{(n)}_k)} \bigoplus_{\pi_\fin \in \Pi_\psi^\fin} m_{\pi_\fin, \psi} \pi_\fin
\]
as a representation of $\Sp_n(\A_\fin)$, 
where the non-negative integer $m_{\pi_\fin, \psi}$ is given by
\[
m_{\pi_\fin, \psi} = \left\{
\begin{aligned}
&1 \iif \pair{\cdot, \pi}_\psi \circ \Delta = \ep_\psi, \\
&0 \other
\end{aligned}
\right.
\]
with setting $\pi = \pi_\fin \otimes (\otimes_{v\mid \infty} \DD^{(n)}_{k_v})$ for $\pi_\fin \in \Pi_\psi^\fin$.
In particular, the space $\Sc_{k}(\Sp_n(\A))$ is multiplicity-free as a representation of $\Sp_n(\A_\fin)$.
Moreover, when $k_v > n$ for any $v \mid \infty$, the inclusion is in fact an equality.
\end{thm}
\begin{proof}
By Arthur's multiplicity formula (Theorem \ref{AMF})
together with the multiplicity one results for $\DD^{(n)}_{k_v}$ in real local $A$-packets 
(Theorem \ref{AJ}, Proposition \ref{MR})
and for the lowest weight vectors in $\DD^{(n)}_{k_v}$, 
the right hand side of the inclusion is the direct sum of 
the subspaces of lowest weight vectors in
$\pi = \otimes_v' \pi_v \subset \AA^2(\Sp_n(\A))$
such that $\pi_v \cong \DD^{(n)}_{k_v}$ for each $v \mid \infty$.
Hence we have the inclusion.
\par

The integer $m_{\pi_\fin, \psi} = m_{\pi, \psi}$ is given in Arthur's multiplicity formula (Theorem \ref{AMF}).
By the multiplicity one results for $p$-adic local $A$-packets (Theorem \ref{mult1}, Proposition \ref{51}), 
we see that the global $A$-packet $\Pi_\psi^\fin$ is a (multiplicity-free) set.
Hence the multiplicity of $\pi_\fin \in \Pi_\psi^\fin$ in $\Sc_{k}(\Sp_n(\A))$ 
is less than or equal to $m_{\pi_\fin, \psi}$.
In particular, $\Sc_{k}(\Sp_n(\A))$ is multiplicity-free.
\par

When $k_v > n$ for any $v \mid \infty$, 
the automorphic representation 
$\pi = \pi_\fin \otimes (\otimes_{v\mid \infty} \DD^{(n)}_{k_v}) \subset \AA^2(\Sp_n(\A))$
is cuspidal by Wallach \cite{W}.
Hence the inclusion is an equality in this case.
\end{proof}

We can now prove the strong multiplicity one theorem (Theorem \ref{SMO}).
Let $\F = \Q$.
\begin{proof}[Proof of Theorem \ref{SMO}]
We consider the cuspidal automorphic representations $\pi_{F_1}, \pi_{F_2} \subset \AA^2(\Sp_n(\A))$
generated by the cusp forms $\varphi_{F_1}$ and $\varphi_{F_2}$, respectively.
By Lemma \ref{irred} and by the assumption, 
they are irreducible and are nearly equivalent to each other.
Hence they belong to the same $A$-packet, say $\Pi_\psi$, by Corollary \ref{near}.
In particular, for each place $v$, the local factors $\pi_{F_1, v}$ and $\pi_{F_2, v}$ are in $\Pi_{\psi_v}$.
\par

When $v = p < \infty$, 
by Propositions \ref{44} and \ref{51}, 
the unramified representations $\pi_{F_1, p}$ and $\pi_{F_2, p}$ are determined uniquely by $\psi_{p}$.
In particular, $\pi_{F_1, p} \cong \pi_{F_2, p}$.
When $v = \infty$, since $\{k_1, k_2\} \not= \{\half{n}, \half{n}+1\}$, by Proposition \ref{MR}, 
we have $k_1 = k_2$ so that $\pi_{F_1, \infty} \cong \pi_{F_2, \infty}$.
Hence we conclude that $\pi_{F_1} \cong \pi_{F_2}$ as $\Sp_n(\A_\fin) \times (\g_\infty, K_\infty)$-modules.
\par

By the multiplicity one result for $\Sc_{k_1}(\Sp_n(\A))$ in Theorem \ref{AMF-S}, 
we have $\pi_{F_1} = \pi_{F_2}$ as a subspace of $\AA^2(\Sp_n(\A))$.
Since both $\varphi_{F_1}$ and $\varphi_{F_2}$ are 
$\Sp_n(\widehat{\Z})$-fixed lowest weight vectors in $\pi_{F_1} = \pi_{F_2}$, 
there exists a constant $c \in \C^\times$ such that $\varphi_{F_2} = c \varphi_{F_1}$. 
This implies that $F_2 = c F_1$.
\end{proof}
\par

\appendix
\section{The hypothesis Langlands group and Arthur's character}\label{arthur-char}
In \S \ref{globalA}, for each global $A$-parameter $\psi \in \Psi_2(\Sp_n/\F)$, 
we have defined Arthur's character $\ep_\psi$ of $A_\psi$.
One might seem that our definition differs from Arthur's original definition.
In this appendix, we formally explain the original definition of Arthur's character 
using the hypothesis Langlands group, 
and show that it agrees with our definition.

\subsection{Hypothesis Langlands group}
Let $\F$ be a number field and $W_\F$ be the Weil group of $\F$.
We denote the set of irreducible unitary cuspidal automorphic representations of $\GL_m(\A)$
by $\AA_\cusp(\GL_m(\A))$.
It is hoped that 
there exists a locally compact group $\LL_\F$ together with a surjection $\LL_\F \twoheadrightarrow W_\F$
such that
it has a canonical bijection
\[
\AA_\cusp(\GL_m(\A))
\longleftrightarrow
\left\{
\text{$m$-dimensional irreducible unitary representations of $\LL_F$}
\right\}.
\]
The group $\LL_F$ is called the {\it hypothesis Langlands group} of $\F$.
Moreover, when $\tau \in \AA_\cusp(\GL_m(\A))$ corresponds to $\phi \colon \LL_F \rightarrow \GL_m(\C)$, 
it is expected that 
$\tau$ is orthogonal (\resp symplectic) if and only if $\phi$ is orthogonal (\resp symplectic).
\par

Using the hypothesis Langlands group $\LL_\F$, 
the global Langlands conjecture roughly states that 
there is a canonical surjection
\[
\xymatrix{
\left\{
\text{irreducible automorphic representations of $\Sp_n(\A)$}
\right\}
\ar@{->>}[d]\\
\left\{
\psi \colon \LL_F \times \SL_2(\C) \rightarrow \SO_{2n+1}(\C)
\right\}.
}
\]
Here we consider a semisimple representation $\psi \colon \LL_F \times \SL_2(\C) \rightarrow \SO_{2n+1}(\C)$
so that $\psi$ decomposes into a direct sum 
\[
\psi = (\phi_1 \boxtimes S_{d_1}) \oplus \dots \oplus (\phi_r \boxtimes S_{d_r})
\]
for some irreducible representations $\phi_1, \dots, \phi_r$ of $\LL_F$.
Replacing $\phi_i$ as corresponding $\tau_i \in \AA_\cusp(\GL_{m_i}(\A))$, 
we obtain the notion of global $A$-parameters (see \S \ref{globalA}).
\par

\subsection{Component groups}
Let $\psi = \oplus_{i=1}^r \phi_i \boxtimes S_{d_i} \colon \LL_F \times \SL_2(\C) \rightarrow \SO_{2n+1}(\C)$.
Set $S_\psi$ and $S_{\psi}^+$ to be 
\[
S_{\psi} = \cent(\im(\psi), \SO_{2n+1}(\C)), 
\quad
S_{\psi}^+ = \cent(\im(\psi), \mathrm{O}_{2n+1}(\C)), 
\]
respectively.
We define the {\it component group} $\Sc_\psi$ and $\Sc_\psi^+$ by
\[
\Sc_\psi = S_\psi/S_\psi^\circ, 
\quad
\Sc_\psi^+ = S_\psi^+/(S_\psi^+)^\circ,
\]
respectively.
We note that $S_\psi^+ = S_\psi \times \{\pm \1_{2n+1}\}$
so that 
$\Sc_\psi^+ = \Sc_\psi \times \{\pm \1_{2n+1}\}$.
\par

Now we assume that $\psi$ is {\it discrete}. 
This means that $\psi$ is a multiplicity-free sum of irreducible orthogonal representations.
Then the image of $\phi_i \boxtimes S_{d_i}$ is contained in an orthogonal group $\mathrm{O}_{m_id_i}(\C)$,
and we have
\[
\Sc_\psi^+ \cong S_\psi^+ = \{\pm \1_{m_1d_1}\} \times \dots \times \{\pm \1_{m_rd_r}\}
\cong \{\pm1\}^r.
\]
By this observation, we obtain the definition of $A_\psi$.
There exists a canonical bijection
\[
A_\psi^+ \longleftrightarrow \Sc_\psi,\ 
\alpha_{\tau_i[d_i]} \longleftrightarrow -\1_{m_id_i}.
\]
In particular, the central element $z_\psi \in A_\psi$ corresponds to $-\1_{2n+1} \in \Sc_\psi^+$, 
which is the center of $\mathrm{O}_{2n+1}(\C)$.

\subsection{Formal definition of Arthur's characters}
Now we explain the definition of Arthur's character formally.
Let $\psi = \oplus_{i=1}^r \phi_i \boxtimes S_{d_i} \colon \LL_F \times \SL_2(\C) \rightarrow \SO_{2n+1}(\C)$
be discrete.
Define a representation 
\[
\Ad_\psi \colon S_\psi \times \LL_\F \times \SL_2(\C) \rightarrow \GL(\so_{2n+1}(\C))
\]
on the Lie algebra $\so_{2n+1}(\C)$ of $\SO_{2n+1}(\C)$ by setting 
\[
\Ad_\psi(s, g, h) = \Ad(s \cdot \psi(g,h))
\]
for $s \in S_\psi$ and $(g,h) \in \LL_\F \times \SL_2(\C)$, 
where $\Ad$ is the adjoint representation of $\SO_{2n+1}(\C)$.
Since $\Ad$ is invariant under Killing form on $\so_{2n+1}(\C)$, 
this representation is orthogonal.
We decompose 
\[
\Ad_\psi = \bigoplus_{\alpha} (\eta_\alpha \boxtimes \phi_\alpha \boxtimes S_{d_\alpha}), 
\]
where $\eta_\alpha$ and $\phi_\alpha$ are irreducible representations of $S_\psi$ and $\LL_\F$, respectively.
After stating \cite[Theorem 1.5.2]{Ar}, 
Arthur defined a character $\ep_\psi$ of $\Sc_\psi = S_\psi$ by 
\[
\ep_\psi(s) = {\prod_{\alpha}}' 
\eta_\alpha(s),
\quad s \in S_\psi, 
\]
where $\prod'_\alpha$ denotes the product over those indices $\alpha$
such that $\phi_\alpha$ is symplectic and $\ep(1/2, \phi_\alpha) = -1$.
Here, when $\phi_\alpha \leftrightarrow \tau_\alpha \in \AA_\cusp(\GL_{m_\alpha}(\A))$, 
we mean that $\ep(1/2, \phi_\alpha) = \ep(1/2, \tau_\alpha)$.
We extend $\ep_\psi$ to $\Sc_\psi^+ = S_\psi^+$ by setting $\ep_\psi(-\1_{2n+1}) = 1$.

\begin{prop}\label{agree}
Let $\psi = \oplus_{i=1}^r \phi_i \boxtimes S_{d_i} \colon \LL_F \times \SL_2(\C) \rightarrow \SO_{2n+1}(\C)$
be discrete, 
where $\phi_i$ is an irreducible representation of $\LL_\F$ 
corresponding to $\tau_i \in \AA_\cusp(\GL_{m_i}(\A))$.
Then we have 
\[
\ep_\psi(-\1_{m_id_i}) = \prod_{j \not= i} \ep \left(\half{1}, \tau_i \times \tau_j \right)^{\min\{d_i,d_j\}}.
\]
\end{prop}
\begin{proof}
We decompose 
$\Ad_\psi = \oplus_{\alpha} (\eta_\alpha \boxtimes \phi_\alpha \boxtimes S_{d_\alpha})$ as above.
Note that $\phi_\alpha$ is symplectic if and only if $d_\alpha$ is even 
since $\eta_\alpha$ is a quadratic character.
Hence, for $s \in S_\psi$, 
\[
\ep_\psi(s) = {\prod_{\alpha}}' \eta_\alpha(s)
= {\prod_{\alpha}}'' \ep \left(\half{1}, \phi_\alpha \right), 
\]
where $\prod''_\alpha$ denotes the product over those indices $\alpha$
such that $d_\alpha$ is even and $\eta_\alpha(s) = -1$.
This means that when we compute $\ep_\psi(s)$, 
we only consider the $(-1)$-eigenspace of $\Ad(s)$, 
which is a subrepresentation of $\Ad \circ \psi = \oplus_{\alpha} (\phi_\alpha \boxtimes S_{d_\alpha})$.
\par

Let $I$ be a subset of $\{1, \dots, r\}$ such that the sum $\sum_{i \in I} m_i d_i$ is even.
Set $s_I = \prod_{i \in I} -\1_{m_i d_i} \in S_\psi$.
Namely $s_I$ acts on $\phi_i \boxtimes S_{d_i}$ by $-1$ if $i \in I$, and by $1$ if $i \not\in I$.
Note that any element in $S_\psi$ is of this form.
\par

Now we compute $\ep_\psi(s_I)$.
First note that 
\[
\Ad \circ \psi \cong \bigoplus_{i} \Ad(\phi_i \boxtimes S_{d_i})
\oplus 
\bigoplus_{i < j} (\phi_i \otimes \phi_j) \boxtimes (S_{d_i} \otimes S_{d_j}).
\]
We see that $\Ad(s_I)$ preserves each summands, and 
\begin{itemize}
\item
$\Ad(s_I)$ acts on $\Ad(\phi_i \boxtimes S_{d_i})$ by $1$; 
\item
$\Ad(s_I)$ acts on $(\phi_i \otimes \phi_j) \boxtimes (S_{d_i} \otimes S_{d_j})$ 
by $-1$ if and only if exactly one of $i$ and $j$ belongs to $I$.
\end{itemize}
On the other hand, 
\[
S_{d_i} \otimes S_{d_j} \cong S_{d_i+d_j-1} \oplus S_{d_i+d_j-3} \oplus \dots \oplus S_{|d_i-d_j|+1}.
\]
In particular, 
$S_{d_i} \otimes S_{d_j}$ has $\min\{d_i,d_j\}$ irreducible summands, 
which are of the form $S_{d_\alpha}$ such that $d_\alpha \equiv d_i+d_j-1 \bmod 2$.
Therefore, by the multiplicativity of $\ep$-factors, we have
\[
\ep_\psi(s_I) = {\prod_{i<j}}''' \ep \left(\half{1}, \tau_i \times \tau_j \right)^{\min\{d_i,d_j\}}, 
\]
where $\prod'''_{i<j}$ denotes the product over those pairs of indices $(i,j)$
such that $i<j$, $d_i \not\equiv d_j \bmod 2$, and 
exactly one of $i$ and $j$ belongs to $I$.
However, \cite[Theorem 1.5.3]{Ar} states that $\ep(1/2, \tau_i \otimes \tau_j) = 1$ if $d_i \equiv d_j \bmod 2$.
Hence one can remove the condition $d_i \not\equiv d_j \bmod 2$ in the definition of $\prod'''_{i<j}$.
In particular, when $I = \{i\}$ (which implies that $m_id_i$ is even), we obtain 
\[
\ep_\psi(-\1_{m_id_i}) = \prod_{j \not= i} \ep \left(\half{1}, \tau_i \times \tau_j \right)^{\min\{d_i,d_j\}}.
\]
\par

Now suppose that $m_id_i$ is odd.
Then we can consider $I = \{1, \dots, r\} \setminus \{i\}$.
Since we define $\ep_\psi(-\1_{2n+1})=1$, we have
\[
\ep_\psi(-\1_{m_id_i}) = 
\ep_\psi(s_I) = \prod_{j \not= i} \ep \left(\half{1}, \tau_i \times \tau_j \right)^{\min\{d_i,d_j\}}.
\]
This completes the proof.
\end{proof}



\begin{thebibliography}{ABCD, 50}

\bibitem{AJ}
{J. Adams and J. F. Johnson}, 
{\em Endoscopic groups and packets of nontempered representations}. 
{\it Compositio Math}. {\bf64} (1987), no.~3, 271--309.

\bibitem{AMR}
{N. Arancibia, C. M{\oe}glin and D. Renard}, 
{\em Paquets d'Arthur des groupes classiques et unitaires}, 
arXiv:1507.01432v2.

\bibitem{Ar}
{J. Arthur},  
{\em The endoscopic classification of representations: orthogonal and symplectic groups},
{\it American Mathematical Society Colloquium Publications}, {\bf 61}, 2013.

\bibitem{AS}
{M. Asgari and R. Schmidt}, 
{\em Siegel modular forms and representations}, 
{\it Manuscripta Math.} {\bf104} (2001), no.~2, 173--200. 

\bibitem{AK}
{H. Atobe and H. Kojima}, 
{\em On the Miyawaki lifts of hermitian modular forms}, 
{\it J. Number Theory} {\bf185} (2018), 281--318.

\bibitem{BJ}
{A. Borel and H. Jacquet}, 
{\em Automorphic forms and automorphic representations}, 
{\it Proc. Sympos. Pure Math}. {\bf 33} 
Part 1, pp. 189--207, Amer. Math. Soc., Providence, R.I., 1979.

\bibitem{H}
{B. Heim}, 
{\em On the modularity of the $\GL_2$-twisted spinor $L$-function}, 
{\it Abh. Math. Semin. Univ. Hambg}. {\bf80} (2010), no.~1, 71--86.

\bibitem{Ib}
{T. Ibukiyama}, 
{\em Lifting conjectures from vector valued Siegel modular forms of degree two}, 
{\it Comment. Math. Univ. St. Pauli} {\bf61} (2012), no.~1, 87--102.

\bibitem{I1}
{T. Ikeda}, 
{\em On the lifting of elliptic cusp forms to Siegel cusp forms of degree $2n$}, 
{\it Ann. of Math. (2)} {\bf154} (2001), no.~3, 641--681.

\bibitem{I2}
{T. Ikeda}, 
{\em Pullback of the lifting of elliptic cusp forms and Miyawaki's conjecture}, 
{\it Duke Math. J.} {\bf131} (2006), no.~3, 469--497. 

\bibitem{JS1}
{H. Jacquet and J. A. Shalika}, 
{\em On Euler products and the classification of automorphic representations. $\mathrm{I}$}, 
{\it Amer. J. Math}. {\bf103} (1981), no.~3, 499--558. 

\bibitem{JS2}
{H. Jacquet and J. A. Shalika}, 
{\em On Euler products and the classification of automorphic forms. $\mathrm{II}$}, 
{\it Amer. J. Math}. {\bf103} (1981), no.~4, 777--815. 

\bibitem{KR}
{P. K. Kewat and R. Raghunathan}, 
{\em On the local and global exterior square $L$-functions of $\GL_n$}, 
{\it Math. Res. Lett.} {\bf19} (2012), no.~4, 785--804. 

\bibitem{M}
{I. Miyawaki}, 
{\em Numerical examples of Siegel cusp forms of degree $3$ and their zeta-functions}, 
{\it Mem. Fac. Sci. Kyushu Univ. Ser. A} {\bf46} (1992), no.~2, 307--339.

\bibitem{Moe1}
{C. M{\oe}glin}, 
{\em Sur certains paquets d'Arthur et involution d'Aubert--Schneider--Stuhler g{\'e}n{\'e}ralis{\'e}e}. 
{\it Represent.~Theory} {\bf10} (2006), 86--129.

\bibitem{Moe2}
{C. M{\oe}glin}, 
{\em Paquets d'Arthur discrets pour un groupe classique $p$-adique}. 
{\it Automorphic forms and $L$-functions $\mathrm{II}$. Local aspects}, 179--257, 
Contemp.~Math., {\bf489}, Israel Math.~Conf.~Proc., 
{\it Amer.~Math.~Soc., Providence, RI}, 2009.

\bibitem{Moe3}
{C. M{\oe}glin}, 
{\em Comparaison des param{\`e}tres de Langlands 
et des exposants {\`a} l'int{\'e}rieur d'un paquet d'Arthur}. 
{\it J.~Lie Theory} {\bf19} (2009), no.~4, 797--840.

\bibitem{Moe4}
{C. M{\oe}glin}, 
{\em Multiplicit{\'e} $1$ dans les paquets d'Arthur aux places $p$-adiques}. 
{\it On certain $L$-functions}, 333--374, 
Clay Math.~Proc., {\bf13}, {\it Amer.~Math.~Soc., Providence, RI}, 2011.

\bibitem{Moe5}
{C.M\oe glin}, 
{\em Image des op{\'e}rateurs d'entrelacements normalis{\'e}s eta p{\^o}les des s{\'e}ries d'Eisenstein}, 
{\it Adv. Math}. {\bf228} (2011), no.~2, 1068--1134.

\bibitem{MR1}
{C. M{\oe}glin and D. Renard}, 
{\em Sur les paquets d'Arthur des groupes classiques r{\'e}els}. 
arXiv:1703.07226v2.

\bibitem{MR}
{C. M{\oe}glin and D. Renard}, 
{\em Sur les paquets d'Arthur de $\Sp(2n,\R)$ contenant des modules unitaires de plus haut poids, scalaires}, 
arXiv:1802.04611v2. 

\bibitem{RS}
{Z. Rudnick and P. Sarnak}, 
{\em Zeros of principal L-functions and random matrix theory}, 
{\it Duke Math. J}. {\bf 81} (1996), 269--322.

\bibitem{W}
{N. R. Wallach}, 
{\em On the constant term of a square integrable automorphic form}, 
{\it Operator algebras and group representations, Vol. II (Neptun, 1980)}, 227--237, 
Monogr. Stud. Math., {\bf18}, {\it Pitman, Boston, MA}, 1984. 

\bibitem{X}
{B. Xu}, 
{\em On M{\oe}glin's parametrization of Arthur packets for $p$-adic quasisplit $\Sp(N)$ and $\SO(N)$}. 
{\it Canad.~J. Math}. {\bf69} (2017), no.~4, 890--960.

\end{thebibliography}
\end{document}